\documentclass{amsart}
\usepackage{amssymb,color,latexsym,amsmath,amsthm}
\usepackage{fullpage}
\renewcommand{\div}{\mathop{\mathrm{div}}\nolimits}
\usepackage[numbers,sort&compress]{natbib}

\numberwithin{equation}{section}
\newtheorem{theorem}{Theorem}[section]
\newtheorem{proposition}[theorem]{Proposition}
\newtheorem{lemma}[theorem]{Lemma}
\newtheorem{corollary}[theorem]{Corollary}

\newtheorem{definition}{Definition}[section]

\newtheorem*{theorem*}{Theorem A}

\newtheorem{thm}{Theorem}[section]

\theoremstyle{definition}
\newtheorem{defn}{Definition}[section]

\newtheorem{remark}{Remark}[section]

\def\ep{\varepsilon}

\def\l{\lambda}
\def\R{{\mathbb R} }

\def\r{\R^{n+1}_{+}}
\def\rn{\R^{n}}
\def\B{\mathbb{R}^n}
\def\br{\partial\r}

\def\s{(-\Delta)^s}

\def\ss{(-\Delta)^{\frac{s}{2}}}
\def\b{\mathbb{S}^{n-1}}
\def\ov{\overline{v}}
\def\ou{\overline{u}}
\def\k{\kappa_{n,s}}
\def\ue{u_e^{\lambda}}
\def\D{\Delta}
\def\Xint#1{\mathchoice
{\XXint\displaystyle\textstyle{#1}}%
{\XXint\textstyle\scriptstyle{#1}}%
{\XXint\scriptstyle\scriptscriptstyle{#1}}%
{\XXint\scriptscriptstyle\scriptscriptstyle{#1}}%
\!\int}
\def\XXint#1#2#3{{\setbox0=\hbox{$#1{#2#3}{\int}$ }
\vcenter{\hbox{$#2#3$ }}\kern-.6\wd0}}

\def\dashint{\Xint-}

\begin{document}
\author[M. Fazly]{Mostafa Fazly}
\address{\noindent Mostafa Fazly, Department of Mathematics, University of Texas at San Antonio, San Antonio TX 78249.}
\email{mostafa.fazly@utsa.edu}

\author[J.C. Wei]{Juncheng Wei}
\address{\noindent Juncheng Wei, Department of Mathematics, University of British Columbia, Vancouver, B.C. Canada V6T 1Z2.}
\email{jcwei@math.ubc.ca}

\author[W. Yang]{ Wen Yang}
\address{\noindent Wen ~Yang,~Wuhan Institute of Physics and Mathematics, Chinese Academy of Sciences, P.O. Box 71010, Wuhan 430071, P. R. China.}
\email{wyang@wipm.ac.cn}

\title{Classification of finite Morse index solutions of higher-order Gelfand-Liouville equation}

\maketitle

\begin{abstract}

 We classify finite Morse index solutions of the following  Gelfand-Liouville equation
 \begin{equation*}
(-\Delta)^{s} u= e^u \ \ \text{in} \ \ \mathbb{R}^n,
\end{equation*}
for $1<s<2$ and $s=2$ via a novel monotonicity formula and technical blow-down analysis.  We show that the above equation does not admit any finite Morse index solution with $\ss u$ vanishes at infinity provided $n>2s$ and
\begin{equation*}
\label{1.condition}
\frac{ \Gamma^2(\frac{n+2s}{4}) }{ \Gamma^2 (\frac{n-2s}{4})} < \frac{\Gamma(\frac{n}{2}) \Gamma(1+s)}{ \Gamma(\frac{n-2s}{2})},
 \end{equation*}
where $\Gamma$ is the classical Gamma function.  The cases of $s=1$ and $s=2$ are settled by Dancer and Farina \cite{df,d} and  Dupaigne et al.  \cite{dggw}, respectively, using Moser iteration arguments established by Crandall and Rabinowitz \cite{CR}. The case of $0<s<1$ is established by Hyder-Yang in \cite{hy} applying arguments provided in \cite{ddw,fw}.
\end{abstract}



\section{Introduction and Main Results}
We prove a monotonicity formula and apply it to classify stable solutions of the fractional Gelfand-Liouville  equation
\begin{equation}\label{main}
(-\Delta)^{s} u= e^u \ \ \text{in} \ \ \mathbb{R}^n,
\end{equation}
where  $(-\Delta)^{s}$  is the fractional Laplacian operator for $1<s<2$ and for $s=2$.  The above equation has been of great interests in the literature for various parameters $s$ and in multidimensional domains.  For the case of $s=1$, Dancer and Farina in \cite{df}, see also \cite{d}, gave a complete classification of finite Morse index solutions to the above equation. This is connected with the problem known as the Gelfand's problem \cite{gel} on bounded domains.   They showed that in dimensions $3\le n\le 9$, the above equation does not admit any stable solution  outside a compact set. In other words, any solution to (\ref{main})
is unstable outside any compact set, and so its Morse index is infinite. The fourth-order analogue of the Gelfand problem, the case of $s=2$, was settled by Dupaigne et al. \cite{dggw} showing that in dimensions $5\le n\le 12$ any solution is unstable outside every compact set assuming that $\lim_{r\to\infty} \bar v(r)=0$ where
\begin{equation}
v=-\Delta u \ \ \text{and} \ \ \bar v(r)= \dashint_{\partial B_r} v d\sigma.
\end{equation}
A similar assumption was imposed in \cite{wx} to classify (not necessarily stable) solutions of the biharmonic problem in $\mathbb R^4$, see also \cite{lin}. Such assumptions are crucial when dealing with higher-order equations in the context. For the fractional Laplacian operator when $0<s<1$, the authors in \cite{hy} applied a monotonicity formula and blow-down analysis arguments initiated in \cite{ddw,ddww,fw}.  We also refer interested readers to \cite{rss} and references therein where the regularity of extremal solutions is discussed.  In the current article, we are interested in such classification results for finite Morse index solutions when $1<s\leq 2$.

 Assume that $0<s<1$, $u\in C^{2\alpha}(\mathbb{R}^n)\cap L_{s/2}(\mathbb{R}^n)$, $\alpha>s>0$ and $L_{\mu}(\mathbb{R}^n)$ \footnote{One can see that $L_{s/2}(\mathbb{R}^n)\subset L_{s}(\mathbb{R}^n)$ and $\ss u$ is well defined for $u\in L_{s/2}(\mathbb{R}^n)$. If the condition $\ss u$ vanishes at infinity is replaced by $\Delta u$ vanishes at infinity, one can just assume that $u\in L_s(\rn).$} is defined by
\[L_{\mu}(\mathbb{R}^n):=\left\{u\in L_{\mathrm{loc}}^1(\mathbb{R}^n):~ \int_{\mathbb{R}^n} \frac{|u(z)|}{(1+|z|)^{n+2\mu}} dz<\infty\right\}.\]
The fractional Laplacian of $u$
\begin{equation*}
(-\Delta)^s u(x):= p.v. \int_{\mathbb R^n} \frac{u(x)-u(z)}{|x-z|^{n+2s}} dz
\end{equation*}
is well-defined for every $x\in\mathbb R^n$. There are different ways of defining the fractional operator $(-\Delta)^{s}$ when $1<s<2$, just like the case of $0<s<1$. Applying the Fourier transform one can define the fractional Laplacian by
\begin{equation*}
\widehat{ (-\Delta)^{s}}u(\zeta)=|\zeta|^{2s} \hat u(\zeta) ,
\end{equation*}
 or equivalently define this operator inductively by  $(-\Delta)^{s}=  (-\Delta)^{s-1} \text{o} (-\Delta)$.  Yang in \cite{yang} gave a characterization of the fractional Laplacian $(-\Delta)^{s}$, where $s$ is  any positive, non-integer number as the Dirichlet-to-Neumann map for a function $u_e$ satisfying a  higher-order elliptic equation in the upper half space with one extra spatial dimension. This is a generalization of the work of Caffarelli and Silvestre in \cite{cs} for the case of $0<s<1$, see also \cite{cg,clo,cc,li}. Throughout this paper set $b:=3-2s$ and define the operator
\begin{equation}
\Delta_b w:=\Delta w+\frac{b}{y} w_y=y^{-b} \div(y^b \nabla w),
\end{equation}
for a function $w\in W^{2,2}(\mathbb R^{n+1},y^b)$.

\begin{theorem*}
$($\cite{yang}$)$ Let $1<s<2$. For functions $u_e \in W^{2,2} (\mathbb R^{n+1}_+, y^b)$ satisfying the equation
\begin{equation}
\Delta^2_b u_e=0 ,
\end{equation}
 on the upper half space for $(x, y) \in \mathbb R^{n}\times \mathbb R_+$ where
$y$ is the special direction, and the
boundary conditions
\begin{eqnarray*}
 \left\{ \begin{array}{lcl}
\hfill u_e(x,0)&=&f(x) , \\
\hfill \lim_{y\to 0} y^{b}\partial_y{u_e(x,0)}&=& 0 ,
\end{array}\right.
\end{eqnarray*}
along $\{y=0\}$ where $f(x)$ is some function defined on $H^s(\mathbb R^n)$ we have the result that
\begin{equation}
 (-\Delta)^s f(x)= C_{n,s} \lim_{y\to 0} y^{b} \partial_y \Delta_b u_e (x,y) .
 \end{equation}
Moreover,
\begin{equation}
 \int_{\mathbb R^n} |\xi|^{2s} |\hat{u}(\xi)|^2 d\xi=C_{n,s} \int_{\mathbb {R}^{n+1}_+} y^b |\Delta _b u_e(x,y)|^2 dx dy.
  \end{equation}
\end{theorem*}

Applying the above theorem to solutions of (\ref{main}) we conclude that the extended function $u_e(x,y)$ where $x=(x_1,\cdots,x_n)$ and $y\in\mathbb R^+$ satisfies
\begin{eqnarray}\label{maine}
 \left\{ \begin{array}{lcll}
\hfill \Delta^2_b u_e&=& 0   \ \ &\text{in}\ \ \mathbb{R}^{n+1}_{+},\\
\hfill  \lim_{y\to 0} y^{b}\partial_y{u_e}&=& 0   \ \ &\text{on}\ \ \partial\mathbb{R}^{n+1}_{+},\\
\hfill \lim_{y\to 0} y^{b} \partial_y \Delta_b u_e &=& C_{n,s} e^u  \ \ &\text{on}\ \ \partial\mathbb{R}^{n+1}_{+} .
\end{array}\right.
\end{eqnarray}
Moreover,
\begin{equation*}
 \int_{\mathbb R^n} |\xi|^{2s} |\hat{u}(\xi)|^2 d\xi=C_{n,s} \int_{\mathbb {R}^{n+1}_+} y^b |\Delta _b u_e(x,y)|^2 dx dy .
   \end{equation*}
Then $u(x)=u_e(x,0)$.

 \begin{definition}
We say that a solution $u$ of (\ref{main}) is stable outside a compact set  if there exists $R>0$ such that
\begin{equation}\label{stability}
\frac{C_{n,s}}{2} \int_{\mathbb R^n} \int_{\mathbb R^n} \frac{ ( \phi(x)-\phi(z) )^2 }{|x-z|^{n+2s}} dx dz=\int_{\mathbb R^n}|\ss\phi|^2dx \ge  \int_{\mathbb R^n} e^u \phi^2dx ,
\end{equation}
 for any $\phi\in C_c^\infty(\mathbb R^n\setminus \overline {B_{R}})$.
 \end{definition}

We say a solution $u$ has finite Morse index if there is a
finite upper bound on its Morse index in any compact set. This is equivalent to the condition that $u$
is stable outside a compact set.  The main goal of this paper is to classify all solutions of (\ref{main}) which are stable outside a compact set. To this end, we first  introduce the corresponding Joseph-Lungren's exponent, see \cite{jl}.    As it is shown by Herbst in \cite{h} (and also \cite{ya}), for $n>2s$ the following Hardy inequality holds
\begin{equation}
\frac{C_{n,s}}{2} \int_{\mathbb R^n} \int_{\mathbb R^n} \frac{ ( \phi(x)-\phi(z) )^2 }{|x-z|^{n+2s}} dx dz \ge \Lambda_{n,s} \int_{\mathbb R^n} |x|^{-2s} \phi^2 dx,
\end{equation}
for any $\phi \in C_c^\infty(\mathbb R^n)$ where the optimal constant is given by
\begin{equation}
 \Lambda_{n,s}:=2^{2s}\frac{ \Gamma^2(\frac{n+2s}{4})  }{ \Gamma^2(\frac{n-2s}{4})}.
 \end{equation}
We now provide an explicit singular solution for  (\ref{main}). Let $1<s<2$, $n>2s$ and set
\begin{equation}
u_{n,s}(x) := -2s \log |x| + \log A_{n,s} ,
\ \ \text{when} \ \
A_{n,s} := 2^{2s} \frac{\Gamma(\frac{n}{2}) \Gamma(1+s)}{ \Gamma(\frac{n-2s}{2})}.
\end{equation}
Then, $u_{n,s}$  solves (\ref{main}) in $\mathbb R^n\setminus \{0\}$. In order to conclude the above
let $0<t<1$. It is known that
$$ (-\Delta)^t \log \frac{1}{|x|^{2t}} = \frac{A_{n,t}}{|x|^{2t}},
\ \text{when} \ \
A_{n,t} :=  2^{2t} \frac{\Gamma(\frac{n}{2}) \Gamma(1+t)}{ \Gamma(\frac{n-2t}{2})} .$$
Since $(-\Delta)^{s} =(-\Delta)o(-\Delta)^{t} $ for $0<t=s-1<1$ we have
\begin{equation*}
(-\Delta)^{s} \log \frac{1}{|x|^{2t}} = -\Delta \left[ \frac{A_{n,t}}{|x|^{2t}} \right] = A_{n,t}
(2t)(n-2t-2)\frac{1}{|x|^{2t+2}}  .
\end{equation*}
Therefore,
\begin{equation*}
(-\Delta)^{s} \log \frac{1}{|x|^{2s-2}} =  A_{n,s-1}
(2s-2)(n-2s)\frac{1}{|x|^{2s}}  .
\end{equation*}
The above implies that
 \begin{equation*}\label{}
(-\Delta)^{s} \log \frac{1}{|x|^{2s}} = A_{n,s-1} 2s(n-2s) \frac{1}{|x|^{2s}}  .
  \end{equation*}
Combining the above coefficients, gives $A_{n,s}$.

The above computations imply that equation (\ref{main}) may
not admit any stable solution if $n>2s$ and $ \Lambda_{n,s} < A_{n,s}$ that reads
\begin{equation}
\label{1.condition}
\frac{ \Gamma^2(\frac{n+2s}{4})  }{ \Gamma^2(\frac{n-2s}{4})} < \frac{\Gamma(\frac{n}{2}) \Gamma(1+s)}{ \Gamma(\frac{n-2s}{2})} .
 \end{equation}
 In general, the above  implies that $2s<n<n_0(s)$ for some $n_0(s)\in\mathbb R^+$. It is straightforward to notice that for the case of $s=1$ the above  reads $2<n<n_0(1)$ when $n_0(1)=10$ and for the case of $s=2$, it reads $4<n<n_0(2)$ when $n_0(2)\approx 12.56$ is the largest positive root of
$  n^2(n-4) -128 (n-2) $. We now elaborate on the above inequality when $1<s<2$. Let $f:[2s,\infty)\to\mathbb R^+$ be the following continuous function
\begin{equation*}
f(x):=\frac{ \Gamma^2(\frac{x+2s}{4}) }{ \Gamma^2(\frac{x-2s}{4}) } \frac{   \Gamma(\frac{x-2s}{2}) }{ \Gamma(\frac{x}{2}) } .
 \end{equation*}
This implies that $f(2s)= \Gamma(s)$. Now, define $g(x):=f(x) - \Gamma(1+s) $. So, $$g(2s)=\Gamma(s) - \Gamma(1+s) = (1-s)\Gamma(s)<0 \ \text{when} \ s>1.$$
 Applying the Stirling's formula for the gamma function one can observe that
\begin{equation}
f(x) \approx O(x^s) \ \ \text{when} \ x\  \text{is large}.
\end{equation}
This implies that when $x$ is large enough $g(x)$ becomes positive. We define $n_0(s)$ to be the first positive root of $g(x)$ when $x>2s$. We therefore conclude that $n_0(s)$ satisfies  $10=n_0(1)<n_0(s)<n_0(2)\approx 12.5$ when $1<s<2$.

Here is our main result.

\begin{thm}
\label{main-th}
Assume that $n>2s$ and $1<s<\alpha<2$. Let $u\in C^{2\alpha}(\rn)\cap L_{s/2}(\rn)$ be a solution of \eqref{main} which is stable outside a compact set and $\ss u$ vanishes at infinity. Then
$$\frac{ \Gamma^2 (\frac{n+2s}{4}) }{ \Gamma^2(\frac{n-2s}{4})} \geq \frac{\Gamma(\frac{n}{2}) \Gamma(1+s)}{ \Gamma(\frac{n-2s}{2})}.$$
\end{thm}


This article is organized as follows. In Section \ref{secmf}, we prove a monotonicity formula via rescaling arguments. In Section \ref{sechs}, we classify the homogeneous solutions. In Section \ref{secbd}, we perform a blow-down analysis argument. To do so, we start with some preliminary integral estimates for stable solutions.  Then, we perform a bootstrap argument and Moser iteration type arguments. Applying these results we provide technical estimates for each term in the monotonicity formula. This enables us to prove our main result.  In Section \ref{secfo}, we study the fourth order Gelfand-Liouville equation. 

\bigskip
	\begin{center}
		Notations:
	\end{center}
	\begin{enumerate}
		\item [$B_R^{n+1}$] \quad the ball centered at $0$ with radius $R$ in dimension $(n+1)$.
		\smallskip
		\item [$B_R$] \quad the ball centered at $0$ with radius $R$ in dimension $n$.
		\smallskip
		\item [$B^{n+1}(x_0,R)$] \quad  the ball centered at $x_0$ with radius $R$ in dimension $(n+1)$.
		\smallskip
		\item [$B(x_0,R)$] \quad the ball centered at $x_0$ with radius $R$ in dimension $n$.
		\smallskip
		\item [$X=(x,y)$] \quad represent  points in $\mathbb{R}_+^{n+1}=\B\times[0,\infty).$
        \smallskip
		\item [$C$] \quad  a generic positive constant which may change from line to line.
		\smallskip
        \item [$C(r)$] \quad a positive constant depending on $r$ and may change from line to line.
		\smallskip
        \item [$\sigma$] \quad the Hausdorff measure restriced to the boundary of the ball.
	\end{enumerate}

\section{Monotonicity Formula}\label{secmf}
The key technique of our proof is a monotonicity formula that is developed in this section.  The ideas and methods in this section are motivated by the ones in \cite{fw}, see also \cite{fs,ddww,pak}, and references therein.   Define
\begin{eqnarray*}\label{energy}
E(r,x_0,u_e)& :=& r^{2s-n} \left[   \int_{  \mathbb{R}^{n+1}_{+}\cap B^{n+1}(x_0,r)} \frac{1}{2} y^{3-2s}|\Delta_b u_e|^2dxdy-  C_{n,s} \int_{  \partial\mathbb{R}^{n+1}_{+}\cap B^{n+1}(x_0,r)} e^{u_e}dx   \right] \\
&&
-2r^{ 2s-1-n } \int_{  \mathbb{R}^{n+1}_{+}\cap \partial B^{n+1}(x_0,r)}
 y^{3-2s}  \left(  \frac{\partial u_e}{\partial r} +\frac{2s}{r}\right)^2d\sigma
\\&&
+\frac{1}{2}  r^3      \frac{d}{dr} \left[  r^{ 2s-3-n }    \int_{  \mathbb{R}^{n+1}_{+}\cap \partial B^{n+1}(x_0,r)} y^{3-2s} \left(  \frac{\partial u_e}{\partial r} +\frac{2s}{r}    \right )^2 d\sigma    \right]
\\&&
-4s\left(2s-2-n \right) r^{  2s-3 -n }  \int_{  \mathbb{R}^{n+1}_{+}\cap \partial B^{n+1}(x_0,r)}
 y^{3-2s}  \left( u_e + 2s  \log r \right)d\sigma
\\&&
-2s\left(2s-2-n \right) r^{  2s-2 -n }  \int_{  \mathbb{R}^{n+1}_{+}\cap \partial B^{n+1}(x_0,r)}      y^{3-2s} \left(  \frac{\partial u_e}{\partial r} +\frac{2s}{r}    \right)d\sigma
\\&&+ \frac{1}{2}    \frac{d}{dr} \left[ r^{2s-n}  \int_{  \mathbb{R}^{n+1}_{+}\cap \partial B^{n+1}(x_0,r)} y^{3-2s}\left(  | \nabla u_e|^2 - \left|\frac{\partial u_e}{\partial r}\right|^2 \right) d\sigma\right]
\\&&+ \frac{1}{2}    r^{2s-n-1}  \int_{  \mathbb{R}^{n+1}_{+}\cap \partial B^{n+1}(x_0,r)} y^{3-2s} \left(  | \nabla u_e|^2 - \left |  \frac{\partial u_e}{\partial r} \right|^2 \right) d\sigma.
\end{eqnarray*}

\begin{thm}\label{mono}
Assume that $n>2s-1$. Then, $E(\lambda,x,u_e)$ is a nondecreasing function of $\lambda>0$. Furthermore,
\begin{equation}
\frac{dE(\lambda,x_0,u_e)}{d\lambda} \ge C(n,s) \  \lambda^{2s-2-n}    \int_{  \mathbb{R}^{n+1}_{+}\cap \partial B^{n+1}(x_0,\lambda)}  y^{3-2s}\left( \frac{\partial u_e}{\partial r} +  \frac{2s}{ \lambda}  \right)^2d\sigma ,
\end{equation}
where $C(n,s)=2(n+1-2s)$ is independent from $\lambda$.
\end{thm}

\begin{proof}
Suppose that $x_0=0$ and the balls $B_\lambda^{n+1}$ are centered at zero. Set,
\begin{equation}
\bar E(u_e,\lambda):= \lambda^{2s-n} \left(   \int_{  \mathbb{R}^{n+1}_{+}\cap B_\lambda^{n+1}} \frac{1}{2} y^{b} |\Delta_b u_e|^2 dx dy -  C_{n,s}\int_{  \partial\mathbb{R}^{n+1}_{+}\cap B_\lambda^{n+1}} e^{u_e}    dx\right) .
\end{equation}
 Define $v_e:=\Delta_b u_e$, $u_e^\lambda(X):= u_e(\lambda X) +2s \log \lambda$, and $v_e^\lambda(X):=\lambda^{2} v_e(\lambda X)$ where $X=(x,y)\in\mathbb{R}^{n+1}_+$.  Therefore, $\Delta_b u_e^\lambda(X)=v_e^\lambda(X)$ and
 \begin{eqnarray}\label{mainex}
 \left\{ \begin{array}{lcll}
\hfill \Delta_b v^\lambda_e&=& 0   \ \ &\text{in}\ \ \mathbb{R}^{n+1}_{+},\\
\hfill  \lim_{y\to 0} y^{b}\partial_y{u^\lambda_e}&=& 0   \ \ &\text{on}\ \ \partial\mathbb{R}^{n+1}_{+},\\
\hfill \lim_{y\to 0} y^{b} \partial_y v_e^\lambda &=& C_{n,s} e^{u^\lambda_e}  \ \ &\text{on}\ \ \partial\mathbb{R}^{n+1}_{+} .
\end{array}\right.
\end{eqnarray}
In addition,  differentiating with respect to $\lambda$ we have
 \begin{equation}\label{uvl}
 \Delta_b \frac{du_e^\lambda}{d\lambda}=\frac{dv_e^\lambda}{d\lambda}.
 \end{equation}
Note that
 \begin{equation}
 \bar E(u_e,\lambda)=\bar E(u_e^\lambda,1)= \int_{  \mathbb{R}^{n+1}_{+}\cap B_1^{n+1}} \frac{1}{2} y^b  (v_e^\lambda)^2 dx dy -  C_{n,s} \int_{  \partial\mathbb{R}^{n+1}_{+}\cap B_1^{n+1}} e^{u^\lambda_e}  dx .
  \end{equation}
Taking derivate of the energy with respect to $\lambda$, we have
 \begin{equation}\label{energy}
\frac{d\bar E(u_e^\lambda,1)}{d\lambda}= \int_{  \mathbb{R}^{n+1}_{+}\cap B_1^{n+1}} y^b v_e^\lambda \frac{dv_e^\lambda}{d\lambda}\  dx dy -  C_{n,s} \int_{  \partial\mathbb{R}^{n+1}_{+}\cap B_1^{n+1}} e^{u^\lambda_e}   \frac{du_e^\lambda}{d\lambda} dx.
\end{equation}
Using (\ref{mainex}) we end up with
\begin{equation}\label{energy1}
\frac{d\bar E(u_e^\lambda,1)}{d\lambda}= \int_{  \mathbb{R}^{n+1}_{+}\cap B_1^{n+1}}  y^b v_e^\lambda \frac{dv_e^\lambda}{d\lambda}\  dx dy -  \int_{  \partial\mathbb{R}^{n+1}_{+}\cap B_1^{n+1}}  \lim_{y\to 0} y^{b}  \partial_y v_e^\lambda   \frac{du_e^\lambda}{d\lambda} dx.
\end{equation}
From (\ref{uvl}) and by  integration by parts we have
 \begin{eqnarray*}
 \int_{  \mathbb{R}^{n+1}_{+}\cap B_1^{n+1}} y^b v_e^\lambda \frac{dv_e^\lambda}{d\lambda}dxdy
&=&  \int_{  \mathbb{R}^{n+1}_{+}\cap B_1^{n+1}} y^b \Delta_b u_e^\lambda \Delta_b \frac{du_e^\lambda}{d\lambda}dxdy \\&=&  - \int_{  \mathbb{R}^{n+1}_{+}\cap B_1^{n+1}} \nabla \Delta_b u^\lambda_e \cdot \nabla \left(\frac{du^\lambda_e}{d \lambda}\right) y^bdxdy + \int_{ \partial( \mathbb{R}^{n+1}_{+}\cap B_1^{n+1})} \Delta_b u_e^\lambda y^b \partial_{\nu} \left(   \frac{du^\lambda_e}{d \lambda}  \right)d\sigma .
\end{eqnarray*}
Note that
 \begin{eqnarray*}
-\int_{  \mathbb{R}^{n+1}_{+}\cap B_1^{n+1}} \nabla \Delta_b u_e\cdot \nabla \frac{du^\lambda_e}{d \lambda} y^bdxdy &=& \int_{  \mathbb{R}^{n+1}_{+}\cap B_1^{n+1}} \div( \nabla\Delta_b u^\lambda_e y^b) \frac{du^\lambda_e}{d \lambda} dxdy-  \int_{ \partial( \mathbb{R}^{n+1}_{+}\cap B_1^{n+1})} y^b   \partial_\nu (\Delta_b u^\lambda_e )  \frac{du^\lambda_e}{d \lambda}d\sigma
\\& =&\int_{  \mathbb{R}^{n+1}_{+}\cap B_1^{n+1}} y^b  \Delta_b^2 u^\lambda_e  \frac{du^\lambda_e}{d \lambda}dxdy -  \int_{ \partial( \mathbb{R}^{n+1}_{+}\cap B_1^{n+1})} y^b   \partial_\nu (\Delta_b u^\lambda_e )  \frac{du^\lambda_e}{d \lambda}d\sigma
\\& =&-  \int_{ \partial( \mathbb{R}^{n+1}_{+}\cap B_1^{n+1})} y^b   \partial_\nu (\Delta_b u^\lambda_e )  \frac{du^\lambda_e}{d \lambda}d\sigma .
\end{eqnarray*}
Therefore,
 \begin{eqnarray*}
 \int_{  \mathbb{R}^{n+1}_{+}\cap B_1^{n+1}} y^b v_e^\lambda \frac{dv_e^\lambda}{d\lambda}dxdy
&=& \int_{ \partial( \mathbb{R}^{n+1}_{+}\cap B_1^{n+1})} \Delta_b u_e^\lambda y^b \partial_{\nu} \left(   \frac{du^\lambda_e}{d \lambda}  \right)d\sigma-  \int_{ \partial( \mathbb{R}^{n+1}_{+}\cap B_1^{n+1})} y^b   \partial_\nu (\Delta_b u^\lambda_e )  \frac{du^\lambda_e}{d \lambda} d\sigma.
\end{eqnarray*}
Boundary of $\mathbb{R}^{n+1}_{+}\cap B_1^{n+1}$ consists of   $\partial\mathbb{R}^{n+1}_{+}\cap B_1^{n+1}$ and  $\mathbb{R}^{n+1}_{+}\cap \partial B_1^{n+1}$. Therefore,
\begin{eqnarray*}
 \int_{  \mathbb{R}^{n+1}_{+}\cap B_1^{n+1}} y^b v_e^\lambda \frac{dv_e^\lambda}{d\lambda}dxdy
&=&  \int_{  \partial\mathbb{R}^{n+1}_{+}\cap B_1^{n+1}}\left(- v_e^\lambda \lim_{y\to 0}  y^b \partial_y\left (\frac{du_e^\lambda}{d\lambda}\right) +   \lim_{y\to 0} y^b \partial_y v_e^\lambda \frac{du_e^\lambda}{d\lambda}\right)dx \\&&+ \int_{  \mathbb{R}^{n+1}_{+}\cap \partial B_1^{n+1}} \left(y^b v_e^\lambda \partial_r \left (\frac{du_e^\lambda}{d\lambda}\right) -  y^b \partial_r v_e^\lambda \frac{du_e^\lambda}{d\lambda}\right) d\sigma,
\end{eqnarray*}
where $r=|X|$, $X=(x,y)\in \mathbb{R}^{n+1}_+$ and $\partial_r=\nabla\cdot \frac{X}{r}$ is the corresponding radial derivative.   Note that the first integral in the right-hand side vanishes since $\partial_y\left (\frac{du_e^\lambda}{d\lambda}\right)=0$ on $ \partial\mathbb{R}^{n+1}_{+}$. From (\ref{energy}) we obtain
 \begin{eqnarray}\label{energy2}
\frac{d\bar E(u_e^\lambda,1)}{d\lambda}= \int_{  \mathbb{R}^{n+1}_{+}\cap \partial B_1^{n+1}}  y^b\left ( v_e^\lambda \partial_r \left (\frac{du_e^\lambda}{d\lambda}\right) -  \partial_r \left(v_e^\lambda\right) \frac{du_e^\lambda}{d\lambda} \right)d\sigma .
\end{eqnarray}
Now note that from the definition of $u_e^\lambda$ and $v_e^\lambda$ and by differentiating in $\lambda $ we get the following for $X\in\mathbb{R}^{n+1}_+$
\begin{eqnarray}\label{ulambda}
\lambda \frac{du_e^\lambda(X)}{d\lambda}&=&  r \partial_r u_e^\lambda(X)  + 2s \\\label{vlambda}
\lambda \frac{dv_e^\lambda(X)}{d\lambda}&=&  r \partial_r v_e^\lambda(X) + 2 v_e^\lambda(X) .
\end{eqnarray}
Therefore, differentiating with respect to $\lambda $ we get
\begin{equation*}
\lambda \frac{d^2u_e^\lambda(X)}{d\lambda^2} + \frac{du_e^\lambda(X)}{d\lambda}= r \partial_r \frac{du_e^\lambda(X)}{d\lambda} .
\end{equation*}
So, for all $X\in\mathbb{R}^{n+1}_+\cap \partial B_1^{n+1}$
\begin{eqnarray}
\label{u1lambda}
 \partial_r \left(u_e^\lambda(X)\right)
&=& \lambda \frac{du_e^\lambda(X)}{d\lambda} -2s ,
\\\label{u2lambda}
 \partial_r \left(\frac{du_e^\lambda(X)}{d\lambda}\right)
&= &\lambda \frac{d^2u_e^\lambda(X)}{d\lambda^2} + \frac{du_e^\lambda(X)}{d\lambda} ,
\\\label{v2lambda}
\partial_r \left( v_e^\lambda(X)\right) &=& \lambda \frac{dv_e^\lambda(X)}{d\lambda}- 2 v_e^\lambda(X) .
\end{eqnarray}
Substituting (\ref{u2lambda}) and (\ref{v2lambda}) in (\ref{energy2}) we get
 \begin{eqnarray}\label{energyder}
\frac{d\bar E(u_e^\lambda,1)}{d\lambda}&=&
\int_{  \mathbb{R}^{n+1}_{+}\cap \partial B_1^{n+1}} \left(y^b v_e^\lambda  \left (   \lambda \frac{d^2u_e^\lambda}{d\lambda^2} + \frac{du_e^\lambda}{d\lambda}\right) -   y^b \left( \lambda \frac{dv_e^\lambda}{d\lambda}- 2 v_e^\lambda \right) \frac{du_e^\lambda}{d\lambda}\right)d\sigma
\\&=& \nonumber \int_{  \mathbb{R}^{n+1}_{+}\cap \partial B_1^{n+1}} y^b \left( \lambda  v_e^\lambda    \frac{d^2u_e^\lambda}{d\lambda^2} +3  v_e^\lambda \frac{du_e^\lambda}{d\lambda}    -    \lambda \frac{dv_e^\lambda}{d\lambda} \frac{du_e^\lambda}{d\lambda} \right)d\sigma.
\end{eqnarray}
Taking derivative of (\ref{ulambda}) in $r$ we get
\begin{equation*}
r \frac{\partial^2 u_e^\lambda}{\partial r^2}+ \frac{\partial u_e^\lambda}{\partial r}= \lambda \frac{\partial}{\partial r}\left(\frac{du_e^\lambda}{d\lambda} \right) .
\end{equation*}
So, from (\ref{u2lambda}) for all $X\in\mathbb{R}^{n+1}_+\cap \partial B_1^{n+1}$ we have
 \begin{eqnarray}\label{2ru}
\frac{\partial^2 u_e^\lambda}{\partial r^2} = \lambda \frac{\partial}{\partial r}\left(\frac{du_e^\lambda}{d\lambda} \right) -   \frac{\partial u_e^\lambda}{\partial r} = \lambda \left(  \lambda \frac{d^2u_e^\lambda}{d\lambda^2} + \frac{du_e^\lambda}{d\lambda}  \right) -    \left(  \lambda \frac{du_e^\lambda}{d\lambda} -2s \right)
= \lambda^2 \frac{d^2u_e^\lambda}{d\lambda^2} +2s .
\end{eqnarray}
Note that
 \begin{eqnarray*}
v_e^\lambda= \Delta_b u^\lambda_e = y^{-b} \div(y^b \nabla u^\lambda_e).
\end{eqnarray*}
Set  $\theta_1=\frac {y}{r}$. Then, on $\mathbb{R}^{n+1}_+\cap \partial B_1^{n+1}$, we have
\begin{equation*}
 \div(y^b \nabla u_e^\lambda )=[\partial_{rr} u_e +(n+b)\partial_r u_e ]\theta_1^b +\div_{\mathbb S^n} (\theta_1^b \nabla_{\mathbb{S}^n} u_e^\lambda).
 \end{equation*}
From the above, (\ref{u1lambda}) and (\ref{2ru}) we get
\begin{equation*}
v_e^\lambda = \lambda^2 \frac{d^2u_e^\lambda}{d\lambda^2} +   \alpha \lambda \frac{du_e^\lambda}{d\lambda} +   \beta + \theta_1^{-b}\div_{\mathbb S^n} (\theta_1^b \nabla_{\mathbb{S}^n} u_e^\lambda).
\end{equation*}
where $\alpha:=n + b$ and $\beta:=2s\left(1-n-b \right)$. From this and (\ref{energyder}) we get
 \begin{eqnarray*}
\frac{d\bar E(u_e^\lambda,1)}{d\lambda}&=& \int_{  \mathbb{R}^{n+1}_{+}\cap \partial B_1^{n+1}} \theta_1^b \lambda  \left(    \lambda^2 \frac{d^2u_e^\lambda}{d\lambda^2} +\alpha \lambda \frac{du_e^\lambda}{d\lambda} + \beta      \right)   \frac{d^2u_e^\lambda}{d\lambda^2}d\sigma
+  \int_{  \mathbb{R}^{n+1}_{+}\cap \partial B_1^{n+1}}   3\theta_1^b \left(    \lambda^2 \frac{d^2u_e^\lambda}{d\lambda^2} +\alpha \lambda \frac{du_e^\lambda}{d\lambda} + \beta     \right)          \frac{du_e^\lambda}{d\lambda}d\sigma
\\&&- \int_{  \mathbb{R}^{n+1}_{+}\cap \partial B_1^{n+1}}     \theta_1^b  \lambda  \frac{du_e^\lambda}{d\lambda}  \frac{d}{d\lambda}  \left(    \lambda^2 \frac{d^2u_e^\lambda}{d\lambda^2} +\alpha \lambda \frac{du_e^\lambda}{d\lambda} + \beta     \right)d\sigma
 +\int_{  \mathbb{R}^{n+1}_{+}\cap \partial B_1^{n+1}} \theta_1^b \lambda   \frac{d^2u_e^\lambda}{d\lambda^2} \theta_1^{-b}\div_{\mathbb S^n} (\theta_1^b \nabla_{\mathbb{S}^n} u_e^\lambda)d\sigma
 \\&& +\int_{  \mathbb{R}^{n+1}_{+}\cap \partial B_1^{n+1}}  3  \theta_1^b \ \frac{du_e^\lambda}{d\lambda}  \theta_1^{-b}\div_{\mathbb S^n} (\theta_1^b \nabla_{\mathbb{S}^n} u_e^\lambda)d\sigma
 -  \int_{  \mathbb{R}^{n+1}_{+}\cap \partial B_1^{n+1}} \theta_1^b  \lambda \frac{d}{d\lambda}\left(  \theta_1^{-b}\div_{\mathbb S^n} (\theta_1^b \nabla_{\mathbb{S}^n} u_e^\lambda)  \right) \frac{du_e^\lambda}{d\lambda} d\sigma.
\end{eqnarray*}
 Simplifying the integrals we get
\begin{eqnarray*}
\frac{d\bar E(u_e^\lambda,1)}{d\lambda}&=&
 \int_{  \mathbb{R}^{n+1}_{+}\cap \partial B_1^{n+1}}  \theta_1^b \left(  \lambda^3    \left(   \frac{d^2u_e^\lambda}{d\lambda^2}\right)^2 +  \lambda^2   \frac{du_e^\lambda}{d\lambda} \frac{d^2u_e^\lambda}{d\lambda^2}
+ 2 \alpha \lambda  \left(   \frac{d u_e^\lambda}{d\lambda }\right)^2 \right)d\sigma
\\&& +
 \int_{  \mathbb{R}^{n+1}_{+}\cap \partial B_1^{n+1}}  \theta_1^b \left(
 -   \lambda^3  \frac{du_e^\lambda}{d\lambda} \frac{d^3 u_e^\lambda}{d\lambda^3}
+ 3 \beta  \frac{du_e^\lambda}{d\lambda}  +\beta \lambda \frac{d^2u_e^\lambda}{d\lambda^2}  \right)d\sigma
\\&& \nonumber+\int_{  \mathbb{R}^{n+1}_{+}\cap \partial B_1^{n+1}} \left(\lambda \frac{d^2u_e^\lambda}{d\lambda^2} \div_{\mathbb S^n} (\theta_1^b \nabla_{\mathbb{S}^n} u_e^\lambda)  +3  \div_{\mathbb S^n} (\theta_1^b \nabla_{\mathbb{S}^n} u_e^\lambda)    \frac{du_e^\lambda}{d\lambda}    -    \lambda \frac{d}{d\lambda}\left(  \div_{\mathbb S^n} (\theta_1^b \nabla_{\mathbb{S}^n} u_e^\lambda)   \right) \frac{du_e^\lambda}{d\lambda}\right)d\sigma .
\end{eqnarray*}
Note that
\begin{eqnarray*}
\lambda^2   \frac{du_e^\lambda}{d\lambda} \frac{d^2u_e^\lambda}{d\lambda^2}
&=& \frac{1}{2}  \frac{d}{d\lambda} \left[ \lambda^2   \left(   \frac{d u_e^\lambda}{d\lambda}\right)^2 \right] - \lambda
 \left(   \frac{d u_e^\lambda}{d\lambda }\right)^2
\\
\lambda   \frac{d^2u_e^\lambda}{d\lambda^2}
&=&   \frac{d}{d\lambda} \left[ \lambda   \frac{d u_e^\lambda}{d\lambda} \right] -  \frac{d u_e^\lambda}{d\lambda }
\\
\lambda^3   \frac{du_e^\lambda}{d\lambda} \frac{d^3 u_e^\lambda}{d\lambda^3}
&=&  \frac{1}{2}  \frac{d}{d\lambda} \left[ \lambda^3   \frac{d}{d\lambda} \left[ \left(   \frac{d u_e^\lambda}{d\lambda}\right)^2  \right] \right]
-\frac{3}{2} \lambda^2  \frac{d}{d\lambda}  \left[ \left(   \frac{d u_e^\lambda}{d\lambda}\right)^2  \right]
-  \lambda^3 \left(   \frac{d^2 u_e^\lambda}{d\lambda^2}\right)^2 .
\end{eqnarray*}
Combining the above, we conclude
\begin{eqnarray*}
\frac{d\bar E(u_e^\lambda,1)}{d\lambda}&=&
 \int_{  \mathbb{R}^{n+1}_{+}\cap \partial B_1^{n+1}}  \theta_1^b \left[ 2 \lambda^3    \left(   \frac{d^2u_e^\lambda}{d\lambda^2}\right)^2 +2(\alpha-2) \lambda \left(    \frac{du_e^\lambda}{d\lambda}\right)^2 \right]d\sigma
 \\&& \nonumber + \int_{  \mathbb{R}^{n+1}_{+}\cap \partial B_1^{n+1}}  \theta_1^b \left[ 2  \frac{d}{d\lambda} \left( \lambda^2   \left(   \frac{d u_e^\lambda}{d\lambda}\right)^2 \right)  -\frac{1}{2} \frac{d}{d\lambda}  \left(   \lambda^3      \frac{d}{d\lambda} \left(   \frac{d  u_e^\lambda }{d\lambda} \right)^2     \right) +2 \beta    \frac{d  u_e^\lambda }{d\lambda}
+ \beta    \frac{d}{d\lambda} \left( \lambda  \frac{d  u_e^\lambda }{d\lambda} \right)          \right]d\sigma
\\&& \nonumber+\int_{  \mathbb{R}^{n+1}_{+}\cap \partial B_1^{n+1}} \left(\lambda \frac{d^2u_e^\lambda}{d\lambda^2} \div_{\mathbb S^n} (\theta_1^b \nabla_{\mathbb{S}^n} u_e^\lambda)  +3  \div_{\mathbb S^n} (\theta_1^b \nabla_{\mathbb{S}^n} u_e^\lambda)    \frac{du_e^\lambda}{d\lambda}    -    \lambda \frac{d}{d\lambda}\left(  \div_{\mathbb S^n} (\theta_1^b \nabla_{\mathbb{S}^n} u_e^\lambda)   \right) \frac{du_e^\lambda}{d\lambda}\right)d\sigma
\\&=&:R_1+R_2+R_3.
\end{eqnarray*}
Note from the assumptions, $\alpha-2>0$. So, $R_1\ge 0$. In order to rewrite $R_2$, we apply the following
\begin{equation}
R_2=\frac{d}{d\lambda} \int_{  \mathbb{R}^{n+1}_{+}\cap \partial B_1^{n+1}}  \theta_1^b \left[ 2   \lambda^2   \left(   \frac{d u_e^\lambda}{d\lambda}\right)^2     -\frac{1}{2}    \lambda^3      \frac{d}{d\lambda} \left(   \frac{d  u_e^\lambda }{d\lambda} \right)^2      +2 \beta    u_e^\lambda
+ \beta   \lambda  \frac{d  u_e^\lambda }{d\lambda}          \right]d\sigma.
\end{equation}
We now apply integration by parts to simplify the terms appeared in $R_3$.
\begin{eqnarray*}
 R_3 &=&  \int_{  \mathbb{R}^{n+1}_{+}\cap \partial B_1^{n+1}}   \left(\lambda \frac{d^2u_e^\lambda}{d\lambda^2} \div_{\mathbb S^n} (\theta_1^b \nabla_{\mathbb{S}^n} u_e^\lambda)  +3  \div_{\mathbb S^n} (\theta_1^b \nabla_{\mathbb{S}^n} u_e^\lambda)    \frac{du_e^\lambda}{d\lambda}    -    \lambda \frac{d}{d\lambda}\left(  \div_{\mathbb S^n} (\theta_1^b \nabla_{\mathbb{S}^n} u_e^\lambda)   \right) \frac{du_e^\lambda}{d\lambda}\right)d\sigma \\
 &=&  \int_{  \mathbb{R}^{n+1}_{+}\cap \partial B_1^{n+1}}\left( - \theta_1^b  \lambda  \nabla_{\mathbb S^n} u_e^\lambda \cdot   \nabla_{\mathbb S^n}  \frac{d^2u_e^\lambda}{d\lambda^2} - 3 \theta_1^b   \nabla_{\mathbb S^n}  u_e^\lambda \cdot    \nabla_{\mathbb S^n}  \frac{du_e^\lambda}{d\lambda}  +  \theta_1^b\lambda \left|     \nabla_{\mathbb S^n}  \frac{du_e^\lambda}{d\lambda}   \right|^2\right)d\sigma \\
&=& - \frac{\lambda}{2}  \frac{d^2}{d\lambda^2} \left(  \int_{  \mathbb{R}^{n+1}_{+}\cap \partial B_1^{n+1}} \theta_1^b  |\nabla_\theta u_e^\lambda |^2  d\sigma \right) -\frac{3}{2} \frac{d}{d\lambda} \left(  \int_{  \mathbb{R}^{n+1}_{+}\cap \partial B_1^{n+1}}  \theta_1^b |\nabla_\theta u_e^\lambda |^2   d\sigma\right)+2\lambda \int_{  \mathbb{R}^{n+1}_{+}\cap \partial B_1^{n+1}}  \theta_1^b \left|    \nabla_\theta \frac{du_e^\lambda}{d\lambda}   \right|^2d\sigma
\\&=& - \frac{1}{2}  \frac{d^2}{d\lambda^2} \left(   \lambda      \int_{  \mathbb{R}^{n+1}_{+}\cap \partial B_1^{n+1}} \theta_1^b  |\nabla_\theta u_e^\lambda |^2  d\sigma \right) -\frac{1}{2} \frac{d}{d\lambda} \left(  \int_{  \mathbb{R}^{n+1}_{+}\cap \partial B_1^{n+1}} \theta_1^b  |\nabla_\theta u_e^\lambda |^2  d\sigma \right)+2\lambda \int_{  \mathbb{R}^{n+1}_{+}\cap \partial B_1^{n+1}} \theta_1^b \left|    \nabla_\theta \frac{du_e^\lambda}{d\lambda}   \right|^2d\sigma
\\&\ge& - \frac{1}{2}  \frac{d^2}{d\lambda^2} \left(   \lambda      \int_{  \mathbb{R}^{n+1}_{+}\cap \partial B_1^{n+1}}   \theta_1^b |\nabla_\theta u_e^\lambda |^2   d\sigma\right) -\frac{1}{2} \frac{d}{d\lambda} \left(  \int_{  \mathbb{R}^{n+1}_{+}\cap \partial B_1^{n+1}}  \theta_1^b |\nabla_\theta u_e^\lambda |^2  d\sigma \right) .
 \end{eqnarray*}
Scaling the terms in the above we have
\begin{eqnarray*}
   \frac{d}{d\lambda} \int_{  \mathbb{R}^{n+1}_{+}\cap \partial B_1^{n+1}}
\theta_1^b   \lambda^2 \left(   \frac{d u_e^\lambda}{d\lambda}\right)^2  d\sigma  &=&     \frac{d}{d\lambda} \left[    \lambda^{ 2-b-n } \int_{  \mathbb{R}^{n+1}_{+}\cap \partial B_\lambda^{n+1}}
y^b  \left(\partial_r u_e +\frac{2s}{\lambda}\right)^2 d\sigma  \right] ,
 \\
\frac{d}{d\lambda} \int_{  \mathbb{R}^{n+1}_{+}\cap \partial B_1^{n+1}}
\theta_1^b   \left[  \lambda^3      \frac{d}{d\lambda} \left(   \frac{d  u_e^\lambda }{d\lambda} \right)^2     \right]d\sigma &=&  \frac{d}{d\lambda}  \left[   \lambda^3      \frac{d}{d\lambda} \left(  \lambda^{ -b-n }    \int_{  \mathbb{R}^{n+1}_{+}\cap \partial B_\lambda^{n+1}} y^b \left[ \partial_r u_e +\frac{2s}{\lambda}    \right]^2     \right)  d\sigma \right] ,
 \\
\frac{d}{d\lambda}  \int_{  \mathbb{R}^{n+1}_{+}\cap \partial B_1^{n+1}}  \theta_1^b \lambda  \frac{d  u_e^\lambda }{d\lambda} d\sigma  &=& \frac{d}{d\lambda} \left( \lambda^{  1-b -n }  \int_{  \mathbb{R}^{n+1}_{+}\cap \partial B_\lambda^{n+1}}      y^b \left[ \partial_r u_e +\frac{2s}{\lambda}    \right]  d\sigma  \right).
\end{eqnarray*}
 Note that the two terms that appear as lower bound for $R_3$ are of the form
  \begin{eqnarray*}
 \frac{d^2}{d\lambda^2} \left(   \lambda      \int_{  \mathbb{R}^{n+1}_{+}\cap \partial B_1^{n+1}} \theta_1^b  |\nabla_\theta u_e^\lambda |^2  d\sigma \right) &=&  \frac{d^2}{d\lambda^2} \left[  \lambda^{2s-n}   \int_{  \mathbb{R}^{n+1}_{+}\cap \partial B_\lambda^{n+1}} y^b  \left(  |\nabla  u_e|^2-\left|     \frac{\partial u_e}{\partial r}\right|^2  \right)d\sigma
  \right] ,
  \\
   \frac{d}{d\lambda} \left(  \int_{  \mathbb{R}^{n+1}_{+}\cap \partial B_1^{n+1}}  \theta_1^b |\nabla_\theta u_e^\lambda |^2  d\sigma \right) &=&  \frac{d}{d\lambda} \left[  \lambda^{2s-n-1}     \int_{  \mathbb{R}^{n+1}_{+}\cap \partial B_\lambda^{n+1}} y^b \left(  |\nabla  u_e|^2-\left|     \frac{\partial u_e}{\partial r}\right|^2  \right) d\sigma     \right].
  \end{eqnarray*}
\end{proof}

\section{Homogeneous Solutions}\label{sechs}
In this section, we examine homogeneous solution of the form $\tau(\theta)-2s\log r$. Here we follow the arguments in \cite{fw,hy}.

\begin{thm}
\label{th3.1}	
There is no stable solution of \eqref{main} of the form $\tau(\theta)-2s\log r$ provided \eqref{1.condition} holds.
\end{thm}

\begin{proof}
Since $u$ satisfies \eqref{main}, we have for any radially symmetric function $\varphi\in C_c^\infty(\rn)$
\begin{align*}
\int_{\rn}e^{\tau(\theta)-2s\log|x|}\varphi dx
&=\int_{\rn}(\tau(\theta)-2s\log|x|)(-\Delta)^s\varphi(x)dx\\
&=\int_{\rn}(-2s\log|x|)(-\Delta)^s\varphi(x)dx=A_{n,s}\int_{\rn}\frac{\varphi}{|x|^{2s}}dx,
\end{align*}
where we used
\begin{equation*}
\int_{\rn}\tau(\theta) \s\varphi dx=0\quad \mbox{for any radially symmetric function}~\varphi\in C_c^\infty(\rn) ,
\end{equation*}
and
\begin{equation*}
\s\log\frac{1}{|x|^{2s}}=A_{n,s}\frac{1}{|x|^{2s}}
=2^{2s}\dfrac{\Gamma(\frac{n}{2})\Gamma(1+s)}{\Gamma(\frac{n-2s}{2})}
\frac{1}{|x|^{2s}}.
\end{equation*}
Then,  we derive that
\begin{equation*}
0=\int_{\rn}(e^{\tau(\theta)}-A_{n,s})\frac{\varphi}{|x|^{2s}}dx
=\int_0^\infty r^{n-1-2s}\varphi (r)\int_{\mathbb{S}^{n-1}}(e^{\tau(\theta)}-A_{n,s})d\theta dr,
\end{equation*}
which implies
\begin{equation}
\label{3.int1}
\int_{\mathbb{S}^{n-1}}e^{\tau(\theta)}d\theta=A_{n,s}|\b|.
\end{equation}

We set a radially symmetric smooth cut-off function
\begin{equation*}
\eta(x)=\begin{cases}
1, \qquad &\mathrm{for}~|x|\leq 1,\\
\\
0, &\mathrm{for}~|x|\geq2,
\end{cases}
\end{equation*}
and
$$\eta_\ep(x)=\left(1-\eta\left(\frac{2x}{\ep}\right)\right)\eta(\ep x).$$
It is not difficult to see that $\eta_\ep=1$ for $\ep<r<\ep^{-1}$ and $\eta_\ep=0$ for either $r<\frac{\ep}{2}$ or $r>\frac{2}{\ep}$. We test the stability condition \eqref{stability} on the function $\psi(x)=r^{-\frac{n-2s}{2}}\eta_\ep(r)$. Let $z=rt$ and note that
\begin{equation*}
\begin{aligned}
\int_{\rn}\frac{\psi(x)-\psi(z)}{|x-z|^{n+2s}}dz
&= r^{-\frac n2-s}\int_0^\infty\int_{\b}\frac{\eta_{\ep}(r)-t^{-\frac{n-2s}{2}}\eta_{\ep}(rt)}{(t^2+1-2t\langle \theta,\omega\rangle)^{\frac{n+2s}{2}}}t^{n-1}dtd\omega\\
&= r^{-\frac n2-s}\eta_{\ep}(r)\int_0^{\infty}\int_{\b}\frac{1-t^{-\frac{n-2s}{2}}}
{(t^2+1-2t\langle\theta,\omega\rangle)^{\frac{n+2s}{2}}}t^{n-1}dtd\omega\\
&\quad +r^{-\frac n2-s}\int_0^{\infty}\int_{\b}\frac{t^{n-1-\frac{n-2s}{2}}(\eta_{\ep}(r)-\eta_{\ep}(rt))}
{(t^2+1-2t\langle\theta,\omega\rangle)^{\frac{n+2s}{2}}}dtd\omega.
\end{aligned}
\end{equation*}
It is known that (see e.g.  \cite[Lemma 4.1]{fall})
\begin{equation*}
\Lambda_{n,s}=C_{n,s}\int_0^{\infty}\int_{\b}\frac{1-t^{-\frac{n-2s}{2}}}
{(t^2+1-2t\langle\theta,\omega\rangle)^{\frac{n+2s}{2}}}t^{n-1}dtd\omega.
\end{equation*}
Therefore,
\begin{equation*}
\begin{aligned}
C_{n,s}\int_{\rn}\frac{\psi(x)-\psi(z)}{|x-z|^{n+2s}}dz
=~&C_{n,s}r^{-\frac n2-s}\int_0^{\infty}\int_{\b}\frac{t^{n-1-\frac{n-2s}{2}}(\eta_{\ep}(r)-\eta_{\ep}(rt))}
{(t^2+1-2t\langle\theta,\omega\rangle)^{\frac{n+2s}{2}}}dtd\omega\\
&+\Lambda_{n,s}r^{-\frac n2-s}\eta_{\ep}(r).
\end{aligned}
\end{equation*}
Based on the above computations, we compute the left hand side of the stability inequality \eqref{stability},
\begin{equation}
\label{3.exp-1}
\begin{aligned}
&\frac{C_{n,s}}{2}\int_{\rn}\int_{\rn}\frac{(\psi(x)-\psi(z))^2}{|x-z|^{n+2s}}dxdz
=C_{n,s}\int_{\rn}\int_{\rn}\frac{(\psi(x)-\psi(z))\psi(x)}{|x-z|^{n+2s}}dxdz\\
&=C_{n,s}\int_0^\infty\left[\int_0^\infty r^{-1}\eta_{\ep}(r)(\eta_{\ep}(r)-\eta_{\ep}(rt))dr\right]\int_{\b}\int_{\b}\frac{t^{n-1-\frac{n-2s}{2}}}
{(t^2+1-2t\langle\theta,\omega\rangle)^{\frac{n+2s}{2}}}d\omega d\theta dt\\
&\quad+\Lambda_{n,s}|\b|\int_0^\infty r^{-1}\eta_{\ep}^2(r)dr.
\end{aligned}
\end{equation}
We compute the right hand side of  the stability inequality \eqref{stability} for the test function $\psi(x)=r^{-\frac{n}{2}+s}\eta_{\ep}(r)$ and $u(r)=-2s\log r+\tau(\theta)$,
\begin{equation}
\label{3.exp-2}
\begin{aligned}
\int_{\rn}e^u\psi^2
=\int_0^\infty\int_{\b}\eta_{\ep}^2(r)r^{-2s}r^{-(n-2s)}e^{\tau(\theta)}r^{n-1}drd\theta
=\int_0^\infty r^{-1}\eta^2_{\ep}(r)dr\int_{\b}e^{\tau(\theta)}d\theta.
\end{aligned}
\end{equation}
From the definition of the function $\eta_{\ep}$, we have
\begin{equation*}
\int_0^\infty r^{-1}\eta_{\ep}^2(r)dr=\log\frac{2}{\ep}+O(1).
\end{equation*}
One can see that both the second term on the right hand side of \eqref{3.exp-1} and the right hand side of \eqref{3.exp-2} carry the term $\int_0^\infty r^{-1}\eta_{\ep}^2(r)dr$ and it tends to $\infty$ as $\ep\to0$. Similar computations as in  \cite{hy,fw} yield 
\begin{equation}
\label{3.exp-3}
f_{\ep}(t):=\int_0^\infty r^{-1}\eta_{\ep}(r)(\eta_{\ep}(r)-\eta_{\ep}(rt))dr =O(\log t).
\end{equation}
From this one can see that
\begin{equation*}
\begin{aligned}
&\int_0^{\infty}\left[\int_0^\infty r^{-1}\eta_{\ep}(r)(\eta_{\ep}(r)-\eta_{\ep}(rt))\right]\int_{\b}\int_{\b}
\dfrac{t^{n-1-\frac{n-2s}{2}}}
{(t^2+1-2t\langle\theta,\omega\rangle)^{\frac{n+2s}{2}}}d\omega d\theta dt\\
&\approx \int_0^\infty\int_{\b}\int_{\b}\dfrac{t^{n-1-\frac{n-2s}{2}}\log t}
{(t^2+1-2t\langle\theta,\omega\rangle)^{\frac{n+2s}{2}}}d\omega d\theta dt = O(1).
\end{aligned}
\end{equation*}
Collecting the higher order term $(\log\ep)$, we get
\begin{equation}
\label{3.int2}
\Lambda_{n,s}|\b|\geq \int_{\b}e^{\tau(\theta)}d\theta.
\end{equation}
From \eqref{3.int1} and \eqref{3.int2}, we obtain that $\Lambda_{n,s}\geq A_{n,s},$ which contradicts to the assumption \eqref{1.condition}. Therefore, such homogeneous solution does not exist and we finish the proof.	
\end{proof}

\section{blow-down analysis}\label{secbd}
\subsection{Preliminary estimates}
In this section we use the stability condition outside the compact set to derive a integral representation formula for $u$.
\begin{lemma}
\label{le31.1}
Let $u$ be a solution to \eqref{main} for some $n>2s$. Suppose that $u$ is stable outside a compact set. Then
\begin{equation}
\label{31.1}
\int_{B_r}e^udx\leq Cr^{n-2s},\quad \forall~r\geq1.
\end{equation}
\end{lemma}

\begin{proof}
Let $R\gg1$ be such that $u$ is stable on $\rn\setminus B_R.$ We fix two smooth cut-off functions $\eta_R$ and $\varphi$ in $\rn$ such that
\begin{equation*}
\eta_R(x)=\begin{cases}
0\quad &\mathrm{for}~|x|\leq R,\\
1\quad &\mathrm{for}~|x|\geq 2R,
\end{cases}\qquad
\varphi(x)=\begin{cases}
1\quad &\mathrm{for}~|x|\leq r,\\
0\quad &\mathrm{for}~|x|\geq 2r.
\end{cases}
\end{equation*}
Setting $\psi(x)=\eta_R(x)\varphi(\frac{x}{r})$ with $r\geq1$ we see that $\psi$ is a good test function for stability condition \eqref{stability}. Hence
\begin{equation*}
\int_{B_r}e^udx\leq C+\int_{\rn}|\ss\psi|^2dx\leq Cr^{n-2s},
\end{equation*}
where we used that $n>2s$ and $r\geq1.$
\end{proof}

In the entire paper, we assume that $u$ is stable outside a compact set. We notice that
$$u^\lambda(x)=u(\l x)+2s\log\l$$
provides a family of solutions to \eqref{main}. In addition, $u$ is stable outside a compact set if and only if $u^\l$ is stable outside a compact set. A direct consequence of Lemma \ref{le31.1} is the following $L^1$ estimate.
\begin{corollary}
\label{cr31.1}	
Suppose $n>2s$ and $u$ is a solution of \eqref{main} which is stable outside a compact set. Then there exists $C$ such that
\begin{equation*}
\int_{B_r}e^{u^\l}dx\leq Cr^{n-2s},\quad \forall \l\geq1,~r\geq1.
\end{equation*}
\end{corollary}

With the above lemma in hand, we show that $e^{u^\l}\in L_{\mu}(\rn)$ for some $\mu<0$.

\begin{lemma}
\label{le31.2}
For $\delta>0$ there exists $C=C(\delta)>0$ such that	
\begin{equation}
\int_{\rn}\frac{e^{u^\l}}{1+|x|^{n-2s+\delta}}dx\leq C,\quad \forall \l\geq1.
\end{equation}
\end{lemma}

\begin{proof}
With Corollary \eqref{cr31.1}, we have	
\begin{equation*}
\int_{\rn}\frac{e^{u^\l}}{1+|x|^{n-2s+\delta}}dx
\leq C\int_{B_1}e^{u^\l}dx+
\sum_{i=0}^\infty\int_{2^i\leq|x|<2^{i+1}}\frac{e^{u^\l}}{1+|x|^{n-2s+\delta}}dx\leq C+C\sum_{i=0}^\infty\frac{1}{2^{i\delta}}<+\infty.
\end{equation*}	
Hence we finish the proof.	
\end{proof}
Set
\begin{equation*}
w^\l(x):=c(n,s)\int_{\rn}\left(\frac{1}{|x-z|^{n-2s}}-\frac{1}{(1+|z|)^{n-2s}}\right)e^{u^\l(z)}dz,
\end{equation*}
where $c(n,s)$ is chosen such that
\begin{equation*}
c(n,s)(-\Delta)^s\frac{1}{|x-z|^{n-2s}}=\delta(x-z).
\end{equation*}
It is straightforward to see that $w^\l\in L_{\mathrm{loc}}^1(\rn)$. Here we prove estimates for $w^\l$.

\begin{lemma}
\label{le31.3}
We have
\begin{equation}
\label{31.4}
\int_{\rn}\frac{|w^\l(x)|}{1+|x|^{n+2s}}dx\leq C,\quad \forall \l\geq1.
\end{equation}
Moreover, for every $R>0$ we have
\begin{equation}
\label{31.5}
w^\l(x)=c(n,s)\int_{B_{2R}}\frac{1}{|x-z|^{n-2s}}e^{u^\l(z)}dz+O_R(1)\quad \mbox{for}~\l\geq 1,~|x|\leq R.
\end{equation}
\end{lemma}

\begin{proof}
Let us first set
$$f(x,z)=\left(\frac{1}{|x-z|^{n-2s}}-\frac{1}{(1+|z|)^{n-2s}}\right),$$
and estimate the term
$$E(z):=\int_{\B}\frac{|f(x,z)|}{1+|x|^{n+2s}}dx\quad\mbox{for}\quad |z|\geq 2.$$
We now apply a domain decomposition argument and split $\B$ into
$\B=\bigcup_{i=1}^4A_i,$
where
\begin{equation*}
A_1:=B_{|z|/2},~ A_2:=\B\setminus B_{2|z|},~
A_3:=B(z,|z|/2),~A_4:=\B\setminus (A_1\cup A_2\cup A_3).
\end{equation*}
Then,  it holds that
\begin{equation*}
|f(x,y)|\leq C
\begin{cases}
\dfrac{1+|x|}{|z|^{n-2s+1}} \quad &\mbox{if}~x\in A_1,\\
\dfrac{1}{|z|^{n-2s}}  \quad &\mbox{if}~x\in A_2\cup A_4,\\
\dfrac{1}{|z|^{n-2s}}+\dfrac{1}{|x-z|^{n-2s}} \quad &\mbox{if}~x\in A_3.
\end{cases}
\end{equation*}
We write
\begin{equation*}
E(z)=\sum_{i=1}^4 E_i(z),\quad E_i(z)=\int_{A_i}\dfrac{|f(x,z)|}{1+|x|^{n+2s}}dx.
\end{equation*}	
For $E_1(z)$, one has
\begin{equation*}
|E_1(z)|\leq\dfrac{C}{|z|^{n-2s+1}}\int_{A_1}\dfrac{1+|x|}{1+|x|^{n+2s}}dx\leq \dfrac{C}{|z|^{n-2s+1}}.
\end{equation*}
Next, we bound the second  and fourth terms
\begin{equation*}
|E_2(z)|+|E_4(z)|\leq\frac{C}{|z|^{n-2s}}\int_{\B\setminus A_1}\frac{1}{1+|x|^{n+2s}}dx\leq \frac{C}{|z|^n}.
\end{equation*}
While for the last term $E_3(z)$, we notice that $|x|\sim |z|$ for $x\in A_3$.  Therefore,
\begin{equation*}
\begin{aligned}
|E_3(z)|\leq \frac{C}{|z|^n}+\frac{C}{|z|^{n+2s}}
\int_{A_3}\frac{1}{|x-z|^{n-2s}}dx\leq\frac{C}{|z|^n}+\frac{C}{|z|^{n+2s}}
\int_{B_{3|z|}}\frac{1}{|x|^{n-2s}}dx\leq\frac{C}{|z|^n}.
\end{aligned}
\end{equation*}
Thus,
\begin{equation*}
|E(z)|\leq\dfrac{C}{|z|^{n-2s+1}}\quad \mbox{for}\quad |z|\geq 2.
\end{equation*}
Therefore, by Lemma \ref{le31.2}   we get that
\begin{equation*}
\begin{aligned}
\int_{\B}\dfrac{|w^\l(x)|}{1+|x|^{n+2s}}dx\leq~&
C\int_{\mathbb{R}^n\setminus B_2}e^{u^\l(z)}|E(z)|dz
+C\int_{\B}\int_{B_2}\frac{e^{u^\l(z)}}{|x-z|^{n-2s}(1+|x|^{n+2s})}dzdx\\
&+C\int_{\B}\int_{B_2}\frac{e^{u^\l(z)}}{((1+|z|)^{n-2s})(1+|x|^{n+2s})}dzdx
<+\infty.
\end{aligned}
\end{equation*}	
This finishes the proof of \eqref{31.4}. To prove \eqref{31.5} we notice that
\begin{equation*}
\left|\frac{1}{|x-z|^{n-2s}}-\frac{1}{(1+|z|)^{n-2s}}\right|\leq C\frac{1+|x|}{|z|^{n-2s+1}} \quad\text{for }|x|\leq R,\, |z|\geq 2R.
\end{equation*}
Then
\begin{equation*}
\begin{aligned}
\int_{|z|\geq 2R}\left|\frac{1}{|x-z|^{n-2s}}-\frac{1}{(1+|z|)^{n-2s}}\right|e^{u^\l(z)}dz
\leq C(1+|x|)\int_{|z|\geq2R}\frac{e^{u^\l(z)}}{|z|^{n-2s+1}}dz
\leq C(1+|x|),
\end{aligned}
\end{equation*}
where we used Lemma \ref{le31.2} in the last inequality. Then,  \eqref{31.5} follows immediately.
\end{proof}

\begin{lemma}
\label{le31.5}
We have
\begin{equation}
\label{31.7}
u^\l(x)=c(n,s)\int_{\B}\left(\frac{1}{|x-z|^{n-2s}}-\frac{1}{(1+| z|)^{n-2s}}\right)e^{u^\l(z)}dz+c_\l,
\end{equation}
where $c_\l\in\mathbb{R}$.
\end{lemma}
\begin{proof}
According to the definition of $w^\l$, we can easily see that $h^\lambda:=u^\l-w^\l$ is a $s$-harmonic function in $\B$, i.e., $\s h^\l=0$. In the spirit of   \cite[Lemma 2.4]{ha} one can show that   $h^\l $ is polynomial up to degree 3 at most. In order to prove \eqref{31.7},   we first show
\begin{equation}
\label{31.8}
w^\l(x)\geq -C\log|x|\quad\mbox{for}~|x|~\mbox{large},
\end{equation}
for some constant $C>0$. Indeed,  by \eqref{31.1} we have
\begin{align*}
w^\l(x)\geq -C-C\sum_{i=1}^{[\log(2|x|)]}\int_{2^i\leq |z|\leq 2^{i+1}}\frac{1}{(1+|z|)^{n-2s}}e^{u^\l(z)}dz\geq -C- C\log|x|.
\end{align*}
where $[x]$ denotes the integer part of $x$. Hence, \eqref{31.8} is proved. Therefore
$$u^\l(x)\geq -C\log|x| +h^\l(x).$$
By \eqref{31.1} again, $h^\l  = c_\l+\vec{d}_\l\cdot x+e_\l|x|^2$ for some constants $c_\l\in\mathbb{R}$, $\vec{d}_\l\in\mathbb{R}^n$ and $e_\l<0$. Using the assumption $\ss u(x)$ vanishes at infinity, we conclude that $\vec{d}_\l$ and $e_\l$ vanish. Hence we finish the proof.
\end{proof}


As discussed in the introduction, $v(x)=\ss u(x)$. Let $v^\l(x)=\ss u^\l(x)$ and it is easy to see that $v^\l(x)=\l^sv(\l x).$ Given this and Lemma \ref{le31.3}, we are able to prove the following estimate on $v^\l$.

\begin{lemma}
\label{le31.6}
We have
\begin{equation*}
\int_{\rn} \dfrac{|v^\l(x)|}{1+|x|^{n+s}}dx\leq C,\quad \forall~\l\geq1.
\end{equation*}
\end{lemma}

\begin{proof}
Using Lemma \ref{le31.5} and the assumption $\ss u(x)\to0$ as $x\to+\infty$, we see that
\begin{equation*}
v^\l(x)=C\int_{\rn}\frac{1}{|x-z|^{n-s}}e^{u^\l(z)}dz.
\end{equation*}
Let us set
\begin{equation*}
F(z)=\int_{\rn}\dfrac{1}{|x-z|^{n-s}(1+|x|^{n+s})}dx\quad\mathrm{for}\quad|z|\geq2.
\end{equation*}
We split $\rn$ into
$$\rn= B(z,\frac{|z|}{2})\bigcup \left(\rn\setminus B(z,\frac{|z|}{2})\right).$$
Then, it is easy to check that
\begin{equation*}
\int_{B(z,\frac{|z|}{2})}\dfrac{1}{|x-z|^{n-s}(1+|x|^{n+s})}dx\leq \frac{C}{|z|^n},
\quad\mathrm{and}\quad
\int_{\rn\setminus B(z,\frac{|z|}{2})}\dfrac{1}{|x-z|^{n-s}(1+|x|^{n+s})}dx\leq \frac{C}{|z|^{n-s}}.
\end{equation*}
Thus we have
$$|F(z)|\leq \frac{C}{|z|^{n-s}}\quad \mathrm{for}\quad |z|\geq2.$$
From Lemma \ref{le31.2},  we conclude that
\begin{equation*}
\begin{aligned}
\int_{\rn}\dfrac{|v^\l(x)|}{1+|x|^{n+s}}dx\leq~&
C\int_{\mathbb{R}^n\setminus B_2}e^{u^\l(z)}|F(z)|dz
+C\int_{\rn}\int_{B_2}\frac{e^{u^\l(z)}}{|x-z|^{n-s}(1+|x|^{n+s})}dzdx\\
\leq~&C+ C\int_{B_2}e^{u^\l(z)}dz<\infty.
\end{aligned}
\end{equation*}
This finishes the proof.	
\end{proof}

\subsection{A bootstrap argument}
The aim of this subsection is devoted to of deriving the higher order energy estimate, which will be carried out through a bootstrap argument. The main idea of the argument is known as the Moser iteration type arguments initiated by Crandall and Rabinowitz in \cite{CR} and adapted in \cite{dggw}.

\begin{lemma}
\label{le32.1}
Let the stability inequality \eqref{stability} holds. Then, for $t\in[0,1],$
\begin{equation*}
\int_{\mathbb{R}^n}e^{tu}\varphi^2dx\leq\int_{\rn}|(-\Delta)^{\frac{st}{4}}\varphi|^2dx,\quad \forall~\varphi\in C_c^\infty(\rn).
\end{equation*}
In particular, for $t=\frac12,$
\begin{equation}
\label{half-stability}
\int_{\mathbb{R}^n}e^{\frac12u}\varphi^2dx\leq\int_{\rn}|(-\Delta)^{\frac{s}{4}}\varphi|^2dx,\quad \forall~\varphi\in C_c^\infty(\rn).
\end{equation}
\end{lemma}

\begin{proof}
We apply complex interpolation between the family of spaces $X_t,~Y_t$ given for $0\leq t\leq 1$ by
\begin{equation*}
X_t=L^2((2\pi)^{-n}|\xi|^{2st}d\xi), \quad  Y_t=L^2(e^{tu}dx).
\end{equation*}
Recall that the inverse Fourier transform $\mathcal{F}^{-1}: X_0\rightarrow Y_0$ satisfies
$\|\mathcal{F}^{-1}\|_{\mathcal{L}(X_0,Y_0)}=1$. Furthermore, by the stability inequality \eqref{stability} and Plancherel's theorem, we have
\begin{equation*}
\int_{\rn}e^u\varphi^2dx\leq \int_{\rn}|\ss\varphi|^2dx=(2\pi)^{-n}\int_{\rn}|\xi|^{2s}|\mathcal{F}\varphi|^2d\xi.
\end{equation*}
Thus, $\mathcal{F}^{-1}: X_t\rightarrow Y_t$ satisfies $\|\mathcal{F}^{-1}\|_{\mathcal{L}(X_1,Y_1)}\leq 1$. By the complex interpolation theorem, we deduce that $\|\mathcal{F}^{-1}\|_{\mathcal{L}(X_t,Y_t)}\leq 1$ for all $0\leq t\leq1.$
\end{proof}

We shall use the equation \eqref{half-stability} to derive the energy estimate for $e^u$. To do that, we have to study $u,~v$ in the extended setting. Let $v(x)=\ss u(x)$ and $\ou,~\ov$ be the Caffarelli-Silvestre extension of $u,v$, i.e., $\ou,~\ov$ satisfies the following equations respectively
\begin{equation}
\label{32.1}
\begin{cases}
\nabla\cdot(y^{1-s}\nabla\ou)=0   \quad  &\mathrm{in}~\r,\\
\ou=u                              &\mathrm{on}~\partial\r,\\
-\lim_{y\to0}y^{1-s}\partial_y\ou=\kappa_s\ss u= \kappa_s v \quad&\mathrm{on}~\partial\r,
\end{cases}
\end{equation}
and
\begin{equation}
\label{32.2}
\begin{cases}
\nabla\cdot(y^{1-s}\nabla\ov)=0   \quad  &\mathrm{in}~\r,\\
\ov=v                              &\mathrm{on}~\partial\r,\\
-\lim_{y\to0}y^{1-s}\partial_y\ov=\kappa_s\ss v=\kappa_s e^u \quad&\mathrm{on}~\partial\r,
\end{cases}
\end{equation}
where
$$\kappa_s=\frac{\Gamma(1-\frac{s}{2})}{2^{s-1}\Gamma(\frac{s}{2})}.$$ We notice that the stability \eqref{half-stability} can be extended to the extended function. Precisely, if \eqref{half-stability} holds then
\begin{equation*}
\int_{\r}y^{1-s}|\nabla\Phi|^2dxdy\geq \kappa_s\int_{\rn}e^{\frac{u}{2}}\varphi^2dx
\end{equation*}
for every $\Phi\in C_c^\infty(\r)$ satisfying $\varphi(\cdot)=\Phi(\cdot,0)\in C_c^\infty(\rn)$. Indeed, let $\overline{\varphi}$ be the $\frac{s}{2}$-harmonic extension of $\varphi$, we have
\begin{align*}
\int_{\r}y^{1-s}|\nabla\Phi|^2dxdy\geq\int_{\r}y^{1-s}|\nabla\overline{\varphi}|^2dxdy
=\kappa_s\int_{\rn}\varphi\ss\varphi dx\geq\kappa_s\int_{\rn}e^{\frac{u}{2}}\varphi^2dx.
\end{align*}

The following result is about the first step in Moser iteration method.

\begin{lemma}
\label{le32.2}
Assume that $(u,v)$ solves
\begin{equation*}
\begin{cases}
\ss u=v\quad &\mbox{in}~\rn,\\
\ss v=e^u\quad &\mbox{in}~\rn.
\end{cases}
\end{equation*}
Suppose that $u$ is stable in $\Omega\subset\rn.$ Let $\Phi\in C_c^\infty(\r)$ be of the form $\Phi(x,y)=\varphi(x)\eta(y)$ for some $\varphi\in C_c^\infty(\Omega)$ and $\eta\equiv 1$ on $[0,1]$. Take $\alpha>\frac12, \varphi\in C_c^\infty(\Omega)$. Then there exists a constant $C$ depending only on $\alpha$ such that
\begin{equation}
\label{3moser-v}
\frac{\sqrt{2\alpha-1}}{\alpha}\|y^{\frac12-\frac12s}\nabla(\ov^\alpha\Phi)\|_{L^2(\r)}
\leq \kappa_s^\frac12\|e^{\frac{u}{2}}v^{\alpha-\frac12}\varphi\|_{L^2(\rn)}+
C\|y^{\frac12-\frac12s}\ov^{\alpha}\nabla\Phi\|_{L^2(\r)},
\end{equation}
and
\begin{equation}
\label{3moser-u}
\frac{2}{\sqrt{\alpha}}\|y^{\frac12-\frac12s}\nabla(e^{\frac{\alpha}{2}\ou}\Phi)\|_{L^2(\r)}
\leq \kappa_s^\frac12\|e^{\frac{\alpha u}{2}}v^{\frac12}\varphi\|_{L^2(\rn)}+
C\|y^{\frac12-\frac12s}e^{\frac{\alpha}{2}\ou}\nabla\Phi\|_{L^2(\r)}.
\end{equation}
\end{lemma}

\begin{proof}
For the extended equation of $\ov,$ we multiply by $\ov^{2\alpha-1}\varphi^2$ and integrate, we obtain that
\begin{equation*}
\begin{aligned}
\kappa_s\int_{\rn}e^uv^{2\alpha-1}\varphi^2dx=~&\int_{\r}y^{1-s}\nabla\ov\nabla(\ov^{2\alpha-1}\Phi^2)dxdy\\
=~&\frac{2\alpha-1}{\alpha^2}\left(\int_{\r}y^{1-s}|\nabla(\ov^{\alpha}\Phi)|^2dxdy-\int_{\r}y^{1-s}\ov^{2\alpha}|\nabla\Phi|^2dxdy\right)\\
&-\frac{2(\alpha-1)}{\alpha^2}\int_{\r}y^{1-s}\ov^\alpha\Phi\nabla\ov^\alpha\cdot\nabla\Phi dxdy.
\end{aligned}
\end{equation*}
We replace the last term $\Phi\nabla\ov^\alpha$ by $\nabla(\ov^\alpha\Phi)-v^\alpha\nabla\Phi.$ Then,
\begin{equation*}
\begin{aligned}
\kappa_s\int_{\rn}e^uv^{2\alpha-1}\varphi^2dx
=~&\frac{2\alpha-1}{\alpha^2}\int_{\r}y^{1-s}|\nabla(\ov^{\alpha}\Phi)|^2dxdy
-\frac{1}{\alpha^2}\int_{\r}y^{1-s}\ov^{2\alpha}|\nabla\Phi|^2dxdy\\
&-\frac{2(\alpha-1)}{\alpha^2}\int_{\r}y^{1-s}\ov^\alpha\nabla\Phi\nabla(\ov^\alpha\Phi)dxdy,
\end{aligned}
\end{equation*}
which can be rewritten as
\begin{equation}
\label{3.estimate-v}
\begin{aligned}
(2\alpha-1)\int_{\r}y^{1-s}|\nabla(\ov^\alpha\Phi)|^2dxdy
=~&2(\alpha-1)\int_{\r}y^{1-s}\ov^\alpha\nabla\Phi\nabla(\ov^\alpha\Phi)dxdy\\
~&+\int_{\r}y^{1-s}\ov^{2\alpha}|\nabla\Phi|^2dxdy+\alpha^2\kappa_s\int_{\rn}e^uv^{2\alpha-1}\varphi^2dx.
\end{aligned}
\end{equation}
By the Cauchy-Schwarz inequality,
\begin{equation*}
\left|\int_{\r}y^{1-s}\ov^\alpha\nabla\Phi\nabla(\ov^\alpha\Phi)dxdy\right|\leq
\left(\int_{\r}y^{1-s}\ov^{2\alpha}|\nabla\Phi|^2dxdy\right)^{\frac12}\left(\int_{\r}y^{1-s}|\nabla(\ov^\alpha\Phi)|^2dxdy\right)^{\frac12}.
\end{equation*}
Substituting the above equation into \eqref{3.estimate-v}, we obtain a quadratic inequality of the form
$$(2\alpha-1)X^2\leq 2|\alpha-1|AX+A^2+B^2,$$
where
$$X=\|y^{\frac12-\frac12s}\nabla(\ov^\alpha\Phi)\|_{L^2(\r)},\quad A=\|y^{\frac12-\frac12s}\ov^\alpha\nabla\Phi\|_{L^2(\r)},\quad
B=\alpha\|\kappa_s^\frac12e^{\frac12u}v^{\alpha-\frac12}\varphi\|_{L^2(\rn)}.
$$
Solving the quadratic inequality, we deduce that
\begin{equation*}
X\leq\frac{B}{\sqrt{2\alpha-1}}+C_\alpha A.
\end{equation*}
Equation \eqref{3moser-u} follows by the same argument. Particularly, we use the following equality
\begin{equation*}
\frac{4}{\alpha}\int_{\r}y^{1-s}|\nabla(e^{\frac{\alpha}{2}\ou}\Phi)|^2dxdy
=\frac{4}{\alpha}\int_{\r}y^{1-s}\nabla(e^{\frac{\alpha}{2}\ou}\Phi)e^{\frac{\alpha}{2}\ou}\nabla\Phi dxdy
+\kappa_s\int_{\rn}ve^{\alpha u}\varphi^2dx.
\end{equation*}
\end{proof}

Given the above lemma, we  derive the following integral estimates.
\begin{lemma}
\label{le32.3}
Let $\alpha^\sharp,\alpha^*$ denote the largest two roots of the polynomial $X^3-8X+4$. Suppose that $u,v$ satisfies \eqref{32.1}-\eqref{32.2} and \eqref{half-stability} holds for every $\Phi(x,y)=\varphi(x)\eta(y)$ with $\varphi\in C_c^\infty(\Omega)$ and $\eta\equiv 1$ on $[0,1]$. Then, for such $\Phi$ and every $\alpha\in(\alpha^\sharp,\alpha^*)$, there exists a constant $C$ depending only on $\alpha$ such that
\begin{equation}
\label{32.stability-1}
\int_{\r}y^{1-s}|\nabla(\ov^\alpha\Phi)|^2dxdy\leq C\int_{\r}y^{1-s}\ov^{2\alpha}|\nabla\Phi|^2dxdy,
\end{equation}
or
\begin{equation}
\label{32.stability-2}
\int_{\r}y^{1-s}|\nabla(e^{\frac{\alpha}{2}\ou}\Phi)|^2dxdy
\leq C\int_{\r}y^{1-s}e^{\alpha\ou}|\nabla\Phi|^2dxdy.
\end{equation}
\end{lemma}

\begin{proof}
By H\"older's inequality,
\begin{equation*}
\begin{aligned}
\kappa_s\int_{\rn}e^uv^{2\alpha-1}\varphi^2dx
\leq &\left(\kappa_s\int_{\rn}e^{\frac{u}{2}}e^{\alpha u}\varphi^2dx\right)^{\frac{1}{2\alpha}}\left(\kappa_s\int_{\rn}e^{\frac{u}{2}}v^{2\alpha}\varphi^2dx\right)^{\frac{2\alpha-1}{2\alpha}}\\
\leq &\left(\int_{\r}y^{1-s}|\nabla(e^{\frac{\alpha}{2}\ou}\Phi)|^2dxdy
\right)^{\frac{1}{2\alpha}}\left(\int_{\r}y^{1-s}|\nabla(\ov^\alpha\Phi)|^2dxdy\right)^{\frac{2\alpha-1}{2\alpha}}.
\end{aligned}
\end{equation*}
We set
$$H=\|y^{\frac12-\frac12s}\nabla(e^{\frac{\alpha}{2}\ou}\Phi)\|_{L^2(\r)},\quad  K=\|y^{\frac12-\frac12s}\nabla(\ov^\alpha\Phi)\|_{L^2(\r)}.$$
Similarly,
\begin{equation*}
\kappa_s\int_{\rn}e^{\alpha u}v\varphi^2dx\leq K^{\frac{1}{\alpha}}H^{2-\frac{1}{\alpha}}.
\end{equation*}
By Lemma \ref{le32.2}, we get
\begin{equation}
\label{32.hk-1}
\frac{\sqrt{2\alpha-1}}{\alpha}K\leq H^{\frac{1}{2\alpha}}K^{1-\frac{1}{2\alpha}}+ C\|y^{\frac12-\frac12s}\ov^\alpha\nabla\Phi\|_{L^2(\r)},
\end{equation}
\begin{equation}
\label{32.hk-2}
\frac{2}{\sqrt{\alpha}}H\leq
H^{1-\frac{1}{2\alpha}}K^{\frac{1}{2\alpha}}+ C\|y^{\frac12-\frac12s} e^{\frac{\alpha}{2}\ou}\nabla\Phi\|_{L^2(\r)}.
\end{equation}	
Multiplying \eqref{32.hk-1} by \eqref{32.hk-2}. Then
\begin{equation*}
\left(\frac{2\sqrt{2\alpha-1}}{\alpha\sqrt{\alpha}}-1\right)HK
\leq aK^{1-\frac{1}{2\alpha}}H^{\frac{1}{2\alpha}}
+bK^{\frac{1}{2\alpha}}H^{1-\frac{1}{2\alpha}}+ab,
\end{equation*}	
where
$$a=C\|y^{\frac12-\frac12s}e^{\frac{\alpha}{2}\ou}\nabla\Phi\|_{L^2(\r)}\quad\mathrm{and}\quad
b=C\|y^{\frac12-\frac12s}\ov^\alpha\nabla\Phi\|_{L^2(\r)}.$$
It is known that for any $\alpha\in(\alpha^\sharp,\alpha^*)$, we have
$$\delta:=\frac{2\sqrt{2\alpha-1}}{\alpha\sqrt{\alpha}}-1>0.$$	
Introduce
$$X=K^{\frac{1}{2\alpha}}H^{1-\frac{1}{2\alpha}}~\mathrm{and}~Y=K^{1-\frac{1}{2\alpha}}H^{\frac{1}{2\alpha}}.$$	
Then we get
\begin{equation*}
\delta XY\leq aY+bX+ab,
\end{equation*}
and it implies either \eqref{32.stability-1} or \eqref{32.stability-2} holds, we finish the proof.
\end{proof}

We now apply the above lemma and Sobolev inequality arguments to set up a bootstrap procedure.
\begin{lemma}
\label{le32.4}
Let $u$ be the solution of \eqref{main} with \eqref{stability} holds. If there exists a constant $C$ depending only on $\alpha$ such that
\begin{equation}
\label{32.condition}
\int_{B_{2r}\setminus B_r}(e^{\alpha u}+v^{2\alpha})dx\leq Cr^{n-2\alpha s},\quad  \forall r\geq2R.
\end{equation}
Then
\begin{equation}
\label{32.conclusion}
\int_{B_{2r}\setminus B_r}(e^{\frac{n}{n-s}\alpha u}+v^{\frac{2n}{n-s}\alpha})dx\leq Cr^{n-\frac{2n}{n-s}\alpha s},\quad \forall r\geq 4R,
\end{equation}
where $\alpha\in(\alpha^\sharp,\min\{\frac{n}{2s},\alpha^*\}).$
\end{lemma}

\begin{proof}
By \eqref{32.condition}, it is not difficult to see that
\begin{equation*}
\int_{B_r\setminus B_{2R}}e^{\alpha u}\leq Cr^{n-2\alpha s}\quad \mathrm{for}~r>2R.
\end{equation*}
Indeed, for $r>2R$ of the form $r=2^{N_1}$ with some positive integer $N_1$ and taking $N_2$ to be the smallest integer such that $2^{N_2}\geq 2R$, by \eqref{32.condition} we deduce that
\begin{equation}
\label{32.alphau}
\begin{aligned}
\int_{B_{2^{N_1}}\setminus B_{2R}}e^{\alpha u}dx=~&
\int_{B_{2^{N_2}}\setminus B_{2R}}e^{\alpha u}dx+\sum_{\ell=1}^{N_1-N_2}\int_{B_{2^{N_2+\ell}}\setminus B_{2^{N_1+\ell}}}e^{\alpha u}dx\\
\leq~& C+C\sum_{\ell=1}^{N_1-N_2}(2^{N_2+\ell})^{n-2\alpha s}
\leq C2^{N_1(n-2\alpha s)},
\end{aligned}
\end{equation}
where we used $n-2\alpha s>0.$ Next we fix two non-negative smooth functions $\varphi$ on $\rn$ and $\eta$ on $[0,\infty)$ such that
\begin{equation*}
\varphi(x)=\begin{cases}
1\quad &\mathrm{on}~B_2\setminus B_1,\\
\\
0\quad &\mathrm{on}~B_{\frac34}\cup B_3^c,
\end{cases}
\quad\mathrm{and}\quad
\eta(y)=\begin{cases}
1\quad &\mathrm{on}~[0,1],\\
\\
0\quad &\mathrm{on}~[2,+\infty).
\end{cases}
\end{equation*}	
For $r>0$ we set $\Phi_r(x,y)=\varphi(\frac xr)\eta(\frac yr)$. Then we claim that
\begin{equation}
\label{32.claim}
\int_{\r}y^{1-s}(e^{\alpha\ou}+|\ov|^{2\alpha})|\nabla\Phi_r|^2dxdy\leq Cr^{n-(2\alpha+1)s}.
\end{equation}	
To prove the estimation on $e^{\alpha\ou}$, using the hypothesis \eqref{32.condition} we derive the following decay estimate
\begin{equation*}
\begin{aligned}
\int_{|z|\geq r}\frac{e^{\alpha u(z)}}{|z|^{n+s}}dz=
\sum_{i=0}^\infty\int_{2^{i+1}r\geq |z|\geq 2^ir}\frac{e^{\alpha u(z)}}{|z|^{n+s}}dz
\leq\frac{ C}{r^{s+2\alpha s}}\sum_{i=0}^\infty\frac{1}{2^{(1+2\alpha)si}}
\leq Cr^{-s-2\alpha s}.
\end{aligned}
\end{equation*}
On the other hand, by the Poisson representation formula, we have
$$\ou(X)=\int_{\rn}P(X,z)u(z)dz,$$
where $X=(x,y)$ and
\begin{equation*}
P(X,z)=d_{n,s}y^s|X-z|^{-n-s},
\end{equation*}
and $d_{n,s}$ is chosen such that $\int_{\rn}P(X,z)dz=1$. From the expression formula of the Poisson kernel, we get for $|x|\geq 3R$ that
\begin{align*}
\ou(x,y) &\leq C\frac{y^{s}}{(R+y)^{n+s}}\int_{|z|\leq 2R}u^+(z)dz+\int_{\B}\chi_{\B\setminus B_{2R}}(z)u(z)P(X,z)dz\\
&\leq C+\int_{\B}\chi_{\B\setminus B_{2R}}(z)u(z)P(X,z)dz,
\end{align*}
where $\chi_A$ denotes  the characteristic function of a set $A$. Using Jensen's inequality
\begin{align}
\label{31.jensen-1}
e^{\alpha\ou(x,y)}\leq C\int_{\B}\left(e^{\alpha u(z)} \chi_{\B\setminus B_{2R}}(z)+\chi_{B_{2R} }(z) \right)P(X,z)dz \quad\text{for }|x|\geq 3R.
\end{align}
For $r>3R$ and $y\in(0,r)$,  integrating both sides of the inequality \eqref{31.jensen-1} on $B_{r}\setminus B_{3R}$ with respect to $x$
\begin{align*}
\int_{B_{r}\setminus B_{3R}}e^{\alpha\ou(x,y)}dx
&\leq C \int_{|z|\leq 2R}\int_{|x|\leq r}P(X,z)dxdz+C\int_{|z|\geq 2r}\int_{|x|\leq r}e^{\alpha u(z)}P(X,z)dxdz\\
&\quad+C\int_{2R\leq |z|\leq 2r}e^{\alpha u(z)}\int_{|x|\leq r}P(X,z)dxdz \\ &\leq CR^n +Cr^{n+s}\int_{|z|\geq 2r}\frac{e^{\alpha u(z)}}{|z|^{n+s}}dz+C\int_{2R\leq |z|\leq 2r}e^{\alpha u(z)}dz\\
&\leq C+Cr^{n-2\alpha s}\leq Cr^{n-2\alpha s}.
\end{align*}
Using the fact that  $|\nabla \Phi_r |\leq  \frac Cr$, we obtain
\begin{equation}
\label{32.claim1}
\begin{aligned}
\int_{\r}y^{1-s}e^{\alpha\ou}|\nabla  \Phi_r |^2dxdy
\leq Cr^{-2}\int_0^{2r} y^{1-s}\int_{B_{3r}\setminus B_{2r/3}}e^{\alpha\ou(x,y)}dxdy
\leq Cr^{n-s(1+2\alpha)}.
\end{aligned}
\end{equation}	
For the estimation on $\ov$, by the Poisson representation formula, we have
\begin{equation*}
\ov(X)=\int_{\rn}P(X,z)v(z)dz.
\end{equation*}	
Regarding
$d_{n,s}y^s(|z|^2+y^2)^{-\frac12(n+s)}dz$ and $y^{1-s}|\nabla\Phi_r|^2dxdy$ as measures, we apply the Minkowski's inequality to  get
\begin{equation*}
\begin{aligned}
\int_{\r}y^{1-s}\ov^{2\alpha}|\nabla\Phi_r|^2dxdy
\leq &C\left(\int_{B_{8r}}
\left(\int_{\r}y^{1-s}v^{2\alpha}(x-z)|\nabla\Phi_r|^2dxdy\right)^{1/2\alpha}
\frac{y^s}{(|z|^2+y^2)^{\frac{n+s}{2}}}dz\right)^{2\alpha}\\
&+C\int_{\r}y^{1-s}\left(\int_{\rn\setminus B_{8r}}\frac{y^s|v(z)|}{((x-z)^2+y^2)^{\frac{n+s}{2}}}dz\right)^{2\alpha}
|\nabla\Phi_r|^2dxdy.
\end{aligned}
\end{equation*}
Concerning the second term on the right, for $|x|\leq 3r$ and $|y|\leq r$ we get
\begin{equation*}
\begin{aligned}
\int_{\rn\setminus B_{8r}}\frac{y^s|v(z)|}{((x-z)^2+y^2)^{\frac{n+s}{2}}}dz
\leq Cr^s\sum_{i=0}^\infty\int_{B_{2^{i+4}r}\setminus B_{2^{i+3}r}}
\dfrac{v(z)}{|z|^{n+s}}dz
\leq Cr^{-s}\sum_{i=0}^\infty\frac{2^n}{2^{(2i+7)s}}\leq Cr^{-s}.
\end{aligned}
\end{equation*}
The above estimates  yield
\begin{equation}
\label{32.claim2}
\int_{\r}y^{1-s}\ov^{2\alpha}|\nabla\Phi_r|^2dxdy\leq Cr^{n-(2\alpha+1)s}.
\end{equation}
Combining \eqref{32.claim1} and \eqref{32.claim2} we complete the proof of claim in \eqref{32.claim}. Now,  we are ready to prove \eqref{32.conclusion}. By Lemma \ref{le32.3}, either \eqref{32.stability-1} or \eqref{32.stability-2} holds. Assume that \eqref{32.stability-1} holds (and the other case is similar). Using the Sobolev embedding, we obtain
\begin{equation*}
\begin{aligned}
\left(\int_{\rn}v^{\frac{2n}{n-s}\alpha}\varphi^{\frac{2n}{n-s}}dx
\right)^{\frac{n-s}{n}}\leq~& C\int_{\rn}\left|(-\Delta)^{\frac{s}{4}}(v^\alpha\varphi)\right|^2dx
\leq C\int_{\r}y^{1-s}|\nabla(\ov^\alpha\Phi_r)|^2dxdy\\
\leq~&C\int_{\r}y^{1-s}\ov^{2\alpha}|\nabla\Phi_r|^2dxdy
\leq Cr^{n-(2\alpha+1)s}.
\end{aligned}
\end{equation*}
Together with the setting of $\varphi$, we obtain that
\begin{equation*}
\int_{B_{2r}\setminus B_r}v^{\frac{2n}{n-s}\alpha}dx\leq Cr^{n-\frac{2n}{n-s}\alpha s}.
\end{equation*}
Going back to \eqref{32.hk-2}, we deduce that
\begin{equation*}
\begin{aligned}
\left(\int_{\rn}e^{\frac{n}{n-s}\alpha u}\varphi^\frac{2n}{n-s}dx\right)^{\frac{n-s}{n}}\leq~&
C\int_{\r}y^{1-s}|\nabla(e^{\frac{\alpha}{2}\ou}\Phi_r)|^2dxdy\\
\leq~& C\int_{\r}y^{1-s}|\nabla(\ov^\alpha\Phi_r)|^2dxdy
+C\int_{\r}y^{1-s}e^{\alpha\ou}|\nabla\Phi_r|^2dxdy\\
\leq~& Cr^{n-(2\alpha+1)s},
\end{aligned}
\end{equation*}
from which we derived that
\begin{equation*}
\int_{B_{2r}\setminus B_r}e^{\frac{n}{n-s}\alpha u}dx\leq Cr^{n-\frac{2n}{n-s}\alpha s}.
\end{equation*}
Hence we finish the proof.
\end{proof}

Define the following parameter that depends on the dimension $n$ and $s$,
\begin{equation}
\alpha_s=\max\left\{\frac{n}{n-s}\min\left\{\frac{n}{2s},\alpha^*\right\},~\min\left\{\frac{n}{2s},
\alpha^*\right\}+\frac12\right\}.
\end{equation}
Summarizing the prior estimates, we reach the following conclusion for $L^p$ estimates for $e^u$ with certain $p$.

\begin{proposition}
\label{pr32.1}
Let $u$ be a solution of \eqref{main}. Assume that $u$ is stable on $\rn\setminus B_R$. Then for every $p\in[1,\alpha_s)$ there exists $C=C(p)>0$ such that for $r$ large
\begin{equation}
\label{32.est-0}
\int_{B_{2r}\setminus B_r}e^{pu(x)}dx\leq Cr^{n-2ps}.
\end{equation}
In particular,
\begin{itemize}
\item[(i)] for $|x|$ large,
\begin{align}
\label{32.est-1}
\int_{B_{\frac{|x|}{2}}(x)}e^{pu (z)}dz\leq C(p)|x|^{n-2ps},\quad\forall p\in[1,\alpha_s),
\end{align}
\item[(ii)] for $r$ large
\begin{align}
\label{32.est-2}
\int_{B_r\setminus  B_{2R}}e^{pu(x)}dx\leq C(p)r^{n-2ps},\quad \forall p\in[1,\min\{\frac{n}{2s},\alpha^*\}).
\end{align}
\end{itemize}
\end{proposition}

\begin{proof}
By Lemma \ref{le31.1}, we have
\begin{equation}
\label{32.u}
\int_{B_r}e^udx\leq Cr^{n-2s},\quad \forall r\geq1.
\end{equation}
Using the fact that $\ss u(\infty)=0$, we obtain that
$$v(x)=C\int_{\rn}\frac{1}{|x-z|^{n-s}}e^{u(z)}dz.$$
Based on this expression, we see that
\begin{equation*}
\begin{aligned}
\int_{B_r}v(x)dx=~&\int_{B_r}\int_{B_{2r}}\frac{e^{u(z)}}{|x-z|^{n-s}}dzdx
+\int_{B_{r}}\int_{\rn\setminus B_{2r}}\frac{e^{u(z)}}{|x-z|^{n-s}}dzdx\\
\leq~& Cr^{n-s}+Cr^n\sum_{i=1}^{\infty}\int_{B_{2^{i+1}r}\setminus B_{2^ir}}\frac{e^{u(z)}}{|z|^{n-s}}dz\\
\leq~& Cr^{n-s}+Cr^{n-s}\sum_{i=1}^{\infty}\frac{1}{2^{is}}
\leq Cr^{n-s}.
\end{aligned}
\end{equation*}	
Applying standard elliptic estimates,  for $p\in[1,\frac{n}{n-s})$, here we have Lemma \ref{le32.4} holds for $2\alpha>2\alpha^\sharp$ only for $\frac{n}{n-s}>2\alpha^\sharp$. The latter inequality holds automatically under \eqref{1.condition} holds \footnote{
According to our computations $\alpha^\sharp \approx 0.517304$ and $\alpha^* \approx 2.53407$. So, $\frac{n}{n-s}>2\alpha^\sharp$ is equivalent to $n<(1+\frac{1}{2\alpha^\sharp-1})s$ where  $1+\frac{1}{2\alpha^\sharp-1}\approx 29$. Note that we are interested in dimensions in \eqref{1.condition} that reads $2s<n<n_0(s)$ and $10=n_0(1)<n_0(s)<n_0(2) \approx 12.5$ when $1<s<2$. }. So,
\begin{equation*}
\|v\|_{L^p(B_2\setminus B_1)}\leq C\|\ss v\|_{L^1(B_3\setminus B_{\frac{1}{3}})}+C\|v\|_{L^1(B_3\setminus B_{\frac{1}{3}})},
\end{equation*}
and its rescaled version
\begin{equation*}
r^{-n/p}\|v\|_{L^p(B_{2r}\setminus B_r)}\leq Cr^{s-n}\|\ss v\|_{L^1(B_{3r}\setminus B_{\frac{1}{3}r})}
+Cr^{-n}\|v\|_{L^1(B_{3r}\setminus B_{\frac{1}{3}r})}
\leq Cr^{-s},
\end{equation*}
which implies that
\begin{equation}
\label{32.v1}
\int_{B_{2r}\setminus B_r}v^pdx\leq Cr^{n-ps},~\quad \forall r\geq1,~\ p\in[1,\frac{n}{n-s}).
\end{equation}	
Therefore by \eqref{32.u} and \eqref{32.v1} we get that
\begin{equation*}
\int_{B_{2r}\setminus B_r}(e^{\alpha u}+v^{2\alpha})dx\leq Cr^{n-\alpha s},\quad  \forall r\geq 2R,\quad  \forall
\alpha\in(\alpha^\sharp,~\min\{\frac{n}{2n-2s},~1\}).
\end{equation*}	
Repeating the arguments of Lemma \ref{le32.4} finitely many times we could get \eqref{32.est-0}. The reason for us to set $\alpha_s$ is for in the arguments of Lemma \ref{le32.4} we could get
\begin{equation*}
\int_{\rn}e^{(\frac12+2\alpha) u}\varphi^2dx
+\left(\int_{\rn}e^{\frac{n}{n-s}\alpha u}\varphi^\frac{2n}{n-s}dx\right)^{\frac{n-s}{n}}\leq
C\int_{\r}y^{1-s}|\nabla(e^{\frac{\alpha}{2}\ou}\Phi_r)|^2dxdy,
\end{equation*}
which implies
\begin{equation*}
\int_{B_{2r}\setminus B_r}e^{(\frac12+2\alpha)u}dx\leq  C^{r-(2\alpha+1)s}\quad \mathrm{and}\quad  \int_{B_{2r}\setminus B_r}e^{\frac{n}{n-s}\alpha u}dx
\leq Cr^{n-\frac{2n}{n-s}\alpha s}.
\end{equation*}
So we get either $\alpha$ could be improved to $\alpha+\frac12$ or $\frac{n}{n-s}\alpha.$  For \eqref{32.est-1}, it follows immediately from \eqref{32.est-0} as $B_{|x|/2}(x)\subset B_{2r}\setminus B_r$ with $r=|x|$. In the same spirit of the estimate \eqref{32.alphau} one can obtain the \eqref{32.est-2} and we finish the proof.
\end{proof}

We end this subsection with an upper bound on $u$.
\begin{lemma}
\label{le32.5}
Suppose that $n<2\alpha^*s$. If $u$ satisfies  the assumptions in Proposition \ref{pr32.1},   then for $|x|$ large we have
$$u(x)\leq -2s\log|x|+C.$$
\end{lemma}
\begin{proof}
Set
$$W (r):=c(n,s)\int_{|z|\leq r}\frac{1}{(1+|z|)^{n-2s}}e^{u(z)}dz.$$
It is easy to see that $W(r)$ is locally bounded by Lemma \ref{le31.1}. Next,
we claim that for $|x|$ large we have
\begin{align}
\label{u-u}
u(x)=-W(|x|)+O(1).
\end{align}
Using the estimate
$$\left| \frac{1}{|x-z|^{n-2s}}-\frac{1}{(1+|z|)^{n-2s}}\right|\leq C\frac{|x|}{|z|^{n-2s+1}},\quad |x|\geq 1, \,|z|\geq \frac32|x|, $$
we get
\begin{align*}
\int_{|z|\geq \frac32|x|} \left| \frac{1}{|x-z|^{n-2s}}-\frac{1}{(1+|z|)^{n-2s}}\right|e^{u(z)}dz &\leq C|x|\int_{|z|\geq |x|} \frac{e^{u(z)}}{|z|^{n-2s+1}}dz \\&\leq C|x|\sum_{i=0}^\infty\frac{1}{(2^i|x|)^{n-2s+1}}\int_{2^i|x|\leq |z|\leq 2^{i+1}|x|}e^{u(z)}dz\\&\leq C,
\end{align*}
where Lemma \ref{le31.1} is used. Choosing $p\in (1,\alpha_s)$ such that $(n-2s)p'<n$ (this is possible as $n<2\alpha^*s$), and together with \eqref{32.est-1} we have
\begin{align*} \int_{B_\frac{|x|}{2}(x)}\frac{e^{u(z)}}{|x-z|^{n-2s}}dz \leq \left(  \int_{B_\frac{|x|}{2}(x)}\frac{dz}{|x-z|^{(n-2s)p'}} \ \right)^\frac{1}{p'} \left(  \int_{B_\frac{|x|}{2}(x)}e^{pu(z)} dz\ \right)^\frac1p\leq C.
\end{align*}
Using the above estimates and the representation formula \eqref{31.7} we get \eqref{u-u}. Then for $r$ large, using Jensen's inequality we see that
\begin{align*}
\log\left( \frac{1}{|B_r|}\int_{B_r}e^{u(z)}dz\right)
 \geq \frac{1}{|B_r|}\int_{B_r}u(z)dz =-\frac{1}{|B_r|}\int_{B_r}W(|z|)dz+O(1)
 \geq -W(r)+O(1),
\end{align*}
where we used the fact that $W$ is monotone increasing in the last inequality. Thus,  we obtain
$$-W(r)\leq \log(C r^{-2s})+O(1)=-2s\log r+O(1).$$
This completes the proof.
\end{proof}

\subsection{Integral Estimates related to Monotonicity Formula}
First of all, we establish the representation formula of $u_e$. It is discussed in \cite{dghm} that
\begin{equation*}
u_e(x,y)=\kappa_{n,s}\int_{\rn}\frac{y^{2s}}{(y^2+|x-z|^2)^{\frac{n}{2}+s}}u(z)dz=
\int_{\rn}P_{s}(x-z,y)u(z)dz,
\end{equation*}
where
$$P_{s}(x,y)=\kappa_{n,s}\frac{y^{2s}}{(x^2+y^2)^{\frac{n}{2}+s}}\quad\mathrm{and}\quad
\kappa_{n,s}=\frac{\Gamma(\frac{n}{2}+s)}{\Gamma(s)\pi^{\frac{n}{2}}}.$$
According to the setting of $\ue$ and the rescaling argument, one can see that
\begin{equation*}
E(\lambda,0,u_e)=E(1,0,u_e^\lambda).
\end{equation*}
We now estimate the term $y^{3-2s}u_e^\lambda$.
\begin{lemma}
\label{le33.1}
Let $u_e^\l$ be the $s$-harmonic extension of $u^\l$, then
\begin{equation*}
\int_{\partial B_1^{n+1}\cap\r}y^{3-2s}\ue(X)d\sigma=c_sc_\l+O(1),
\end{equation*}
where $c_\l$ is defined in \eqref{31.7} and $c_s$ is a positive finite number given by
$$c_s:=\int_{\partial B_1^{n+1}\cap\r}y^{3-2s}d\sigma.$$
\end{lemma}

\begin{proof}
Using the Poisson's formula we have
\begin{equation*}
\ou^\l(X)=\int_{\B}P_{s}(X,z)u^\l(z)dz=c_\l+\int_{\B}P_{s}(X,z)w^\l(z)dz.
\end{equation*}
It follows from \eqref{31.4} that
\begin{equation*}
\int_{\B\setminus B_2}P_{s}(X,z)|w^\l(z)|dz\leq C\quad\mbox{for}\quad |X|\leq 1.
\end{equation*}
Therefore,
\begin{align*}
\int_{\partial B^{n+1}_1\cap\r}y^{3-2s}\ou^\l(X)d\sigma
=c_sc_\l+O(1)+\int_{\partial B_1^{n+1}\cap\r}y^{3-2s}\int_{B_2}P_{s}(X,z)w^\l(z)dzd\sigma.
\end{align*}
We denote the last term in the above equation by $\Xi $. To estimate $\Xi $, we claim that
\begin{equation}
\label{33.2}
\int_{\partial B^{n+1}_1\cap\r}\int_{B_4}y^{3-2s}P_{s}(X,z)\frac{1}{|\zeta-z|^{n-2s}}dzd\sigma\leq C\quad \mbox{for every}~\zeta\in\B.
\end{equation}
Indeed, for $x\neq \zeta$ we set $R_0=\frac12|x-\zeta|$. Then we have
\begin{equation}
\label{33.3}
\begin{aligned}
&\int_{B_4}\frac{y^3}{|(x-z,y)|^{n+2s}|\zeta-z|^{n-2s}}dz\\
&\leq\left(\int_{B(\zeta,R_0)}+\int_{B_4\setminus  B(\zeta,R_0)}\right)\frac{y^3}{|(x-z,y)|^{n+2s}|\zeta-z|^{n-2s}}dz\\
&\leq C\frac{y^3}{(R_0+y)^{n+2s}}
\int_{B(\zeta,R_0)}\frac{1}{|\zeta-z|^{n-2s}}dz
+\frac{1}{R_0^{n-2s}}\int_{B_4\setminus  B(\zeta,R_0)}\frac{y^3}{|(x-z,y)|^{n+2s}}dz\\
&\leq C\left(\frac{y^3R_0^{2s}}{(R_0+y)^{n+2s}}
+\frac{y^{3-2s}}{R_0^{n-2s}}\right)\leq C\left(\frac{1}{R_0^{n-3}}+\frac{y^{3-2s}}{R_0^{n-2s}}\right).
\end{aligned}
\end{equation}
We use the stereo-graphic projection $(x,y)\to\xi$ from $\partial B_1^{n+1}\cap\R\to\mathbb{R}^n\setminus B_1$, i.e.,
$$(x,y)\to\xi=\frac{x}{1-y}.$$
Then $R_0=\frac12\left|\frac{2\xi}{1+|\xi|^2}-\zeta\right|$ and it follows that
\begin{equation}
\label{33.3333}
\int_{\partial B_1^{n+1}\cap\r}\frac{1}{R_0^{n-3}}d\sigma
\leq\int_{|\xi|\geq 1}\frac{1}{R_0^{n-3}}\frac{1}{(1+|\xi|^2)^n}d\xi\leq C,
\end{equation}
and
\begin{equation}
\label{33.4}
\begin{aligned}
\int_{\partial B_1^{n+1}\cap\r}\frac{y^{3-2s}}{R_0^{n-2s}}d\sigma
\leq \int_{|\xi|\geq1}\frac{y^{3-2s}}{R_0^{n-2s}}
\frac{1}{(1+|\xi|^2)^n}d\xi\leq C.
\end{aligned}
\end{equation}
From \eqref{33.3}-\eqref{33.4}, we proved \eqref{33.2}. As a consequence, we have
\begin{equation*}
\begin{aligned}
|\Xi |&\leq C+C\int_{\partial B^{n+1}_1\cap\r}y^{3-2s}\int_{|z|\leq 2}P_{s}(X,z)\int_{|\zeta|\leq 4}\frac{e^{u^\l(\zeta)}}{|z-\zeta|^{n-2s}}d\zeta dzd\sigma\\
&\leq C+C\int_{|\zeta|\leq 4}e^{u^\l(\zeta)}\int_{\partial B_{1}^{n+1}\cap\r}\int_{|z|\leq 2}y^{3-2s}P_{s}(X,z)\frac{1}{|z-\zeta|^{n-2s}}dzd\sigma d\zeta\\
&\leq C+C\int_{|\zeta|\leq 4}e^{u^\l(\zeta)}d\zeta\leq C,
\end{aligned}
\end{equation*}
where Corollary \ref{cr31.1} is used. Hence, the proof is completed.
\end{proof}

Let $w_e^\l$ stands for the extension of $w^\l$, from the above argument we can also see that $$\left|\int_{\partial B_r^{n+1}\cap\r}y^{3-2s}w_e^\l d\sigma\right| \leq C(r),$$ for some constant $C(r)$ independent of $\l$. In addition, we have the following conclusion. 
\begin{lemma}
\label{le33.12}
Let $w_e^\l$ denote the extension of $w^\l$, then for $\l\geq1$ and $r\geq1$ we have
\begin{equation}
\label{eq.33-1}
\int_{B_r^{n+1}\cap\r}y^{3-2s}|w_e^\l|^2dxdy \leq C(r),
\end{equation}
for some constant $C(r)$ depending only on $r.$
\end{lemma}

\begin{proof}
From the proof of Lemma \ref{le33.1}, it is enough  to show that
\begin{equation}
\label{eq.33-2}
\int_{B_r^{n+1}\cap\r}y^{3-2s}\left(\int_{|z|\leq2r}P_s(X,z)\int_{|\zeta|\leq 4r}\frac{e^{u^\l(\zeta)}}{|\zeta-z|^{n-2s}}d\zeta dz\right)^2dxdy \leq C(r).
\end{equation}
For $x\neq \zeta$, following the arguments in \eqref{33.3} one have
\begin{equation*}
\begin{aligned}
\int_{|z|\leq 2r}\frac{y^{2s}}{|(x-z,y)|^{n+2s}|\zeta-z|^{n-2s}}dz\leq \frac{C}{|x-\zeta|^{n-2s}}.
\end{aligned}
\end{equation*}
As a consequence, we get
\begin{equation}
\label{eq.33-4}
\int_{B_r^{n+1}\cap\r}y^{3-2s}|w_e^\l|^2dxdy\leq C+C\int_{B_r^{n+1}\cap\r}y^{3-2s}\left(\int_{|\zeta|\leq 4r}\frac{e^{u^\l(\zeta)}}{|x-\zeta|^{n-2s}}d\zeta\right)^2dxdy.
\end{equation}
If $2\alpha^*s>4s\geq n>2s$,  from Lemma \ref{le32.5} and $u$ is locally bounded one can get
\begin{equation}
\label{eq.33-5}
\begin{aligned}
\int_{|\zeta|\leq 4r}\frac{e^{u^\l(\zeta)}}{|x-\zeta|^{n-2s}}d\zeta
\leq~& C\int_{\lambda^{-1}M\leq |\zeta|\leq 4r}\frac{1}{|\zeta|^{2s}|x-\zeta|^{n-2s}}d\zeta
+C\int_{ |\zeta|\leq \lambda^{-1}M}\frac{e^{u^\l(\zeta)}}{|x-\zeta|^{n-2s}}d\zeta\\
\leq ~& C(r)+C(M)\left(\frac{M}{\lambda}\right)^{2s}\lambda^{2s}\leq C(r),
\end{aligned}
\end{equation}
where $M$ is a fixed large constant. If $n>4s$, combined with Proposition \ref{pr32.1}-(ii) and $u\in C^{2\alpha}(\R)$, we have
$$\int_{B_r}e^{2u^\l}dx\leq Cr^{n-4s}.$$
From the H\"older's inequality with respect to the measure $\frac{d\zeta}{|x-\zeta|^{n-2s}}$ we get
\begin{equation}
\label{eq.33-6}
\begin{aligned}
\left(\int_{B_{4r}}\frac{e^{u^\l(\zeta)}}{|x-\zeta|^{n-2s}}d\zeta\right)^2
\leq\int_{B_{4r}}\frac{d\zeta}{|x-\zeta|^{n-2s}}\left(\int_{B_{4r}}
\frac{e^{2u^\l(\zeta)}}{|x-\zeta|^{n-2s}}d\zeta\right)
\leq Cr^{2s}\int_{B_{4r}}\frac{e^{2u^\l(\zeta)}}{|x-\zeta|^{n-2s}}d\zeta.
\end{aligned}
\end{equation}
Hence,
\begin{equation}
\label{eq.33-7}
\int_{B_r^{n+1}\cap\r}y^{3-2s}|w_e^\l|^2dxdy \leq C(r)+ C(r)\int_{B_{4r}}e^{2u^\l(\zeta)}
\int_{B_r^{n+1}\cap\r}y^{3-2s}\frac{1}{|x-\zeta|^{n-2s}}dxdyd\zeta
\leq C(r).
\end{equation}
Combining  \eqref{eq.33-4}-\eqref{eq.33-7} we get \eqref{eq.33-1} and it completes the proof.
\end{proof}

To estimate the first quadratic term in the monotonicity formula we need the following $L^2$ estimate for $(-\D)^\frac s2u^\lambda$.

\begin{lemma}
\label{le33.2}
We have
$$\int_{B_r}|(-\D)^\frac s2u^\lambda(x)|^2dx\leq Cr^{n-2s}\quad\text{for~ }r\geq1,~\ \lambda\geq1.$$
\end{lemma}
\begin{proof}
By a scaling argument,  it suffices to prove the lemma for $\lambda=1$. It follows from \eqref{31.7} that
$$(-\D)^\frac s2u(x)=C\int_{\B}\frac{1}{|x-z|^{n-s}}e^{u(z)}dz,$$
in the sense of distribution. Let $R>0$ be such that $u$ is stable outside $B_R$. For $r\gg R$, we decompose $B_{r}=B_{4R}\cup (B_{r}\setminus B_{4R}).$  Since  $u\in C^{2\alpha}(\B)$,  we get
\begin{equation}
\label{33.5}
\int_{B_{4R}}|(-\D)^\frac{s}{2} u|^2dx\leq C(R).
\end{equation}
While for $4R<|x|<r$ we estimate
\begin{equation}
\label{33.6}
\begin{aligned}
|(-\D)^\frac s2u(x)| &\leq C\left(\int_{B_{2R}}+\int_{B_{2r}\setminus B_{2R}}+\int_{\B\setminus B_{2r}}\right)
\frac{1}{|x-z|^{n-s}}e^{u(z)}dz\\
&\leq  \frac{C}{|x|^{n-s}}+C\int_{B_{2r}\setminus B_{2R}}\frac{1}{|x-z|^{n-s}}e^{u(z)}dz+ C\int_{\B\setminus B_{2r}}
\frac{1}{|z|^{n-s}}e^{u(z)}dz \\
&=:C\left(\frac{1}{|x|^{n-s}}+I_1(x)+I_2 \right) .
\end{aligned}
\end{equation}
Using \eqref{31.1} we bound
\begin{align*}
I_2=\sum_{k=1}^\infty \int_{r2^k\leq|x|\leq r2^{k+1}}\frac{e^{u(z)}}{|z|^{n-s}}dz  \leq
C\sum_{k=1}^\infty \frac{(r2^{k+1})^{n-2s}}{(r2^k)^{n-s}}\leq \frac{C}{r^{s}}.
\end{align*}
Therefore,
\begin{equation}
\label{33.7}
\int_{B_r}I_2^2dx\leq Cr^{n-2s}.
\end{equation}
For the second term $I_1(x)$, if $2\alpha^*s>4s\geq n>2s$, by Lemma \ref{le32.5} we have
\begin{equation*}
I_1(x)\leq C\int_{2R\leq|z|\leq 2r}\frac{1}{|x-z|^{n-s}|z|^{2s}}dz\leq \frac{C}{|x|^s},
\end{equation*}
and hence
\begin{equation}
\label{33.8}
\int_{B_r\setminus B_{4R}}I_1^2(x)dx\leq C \int_{B_r\setminus B_{4R}}\frac{1}{|x|^{2s}}dx\leq Cr^{n-2s}.
\end{equation}
If $n>4s$,  by H\"older's inequality with respect to the measure $\frac{dz}{|x-z|^{n-s}}$  we get
\begin{align*}
\left(\int_{B_{2r}\setminus B_{2R}}\frac{e^{u(z)}}{|x-z|^{n-s}}dz\right)^2
\leq\int_{B_{2r}}\frac{dz}{|x-z|^{n-s}}\left(\int_{B_{2r}\setminus B_{2R}}\frac{e^{2u(z)}}{|x-z|^{n-s}}dz\right)
\leq Cr^s\int_{B_{2r}\setminus B_{2R}}\frac{e^{2u(z)}}{|x-z|^{n-s}}dz.
\end{align*}
Hence by Proposition \ref{pr32.1}, we get  that \eqref{33.8} also holds for $n>4s$.
Combining  \eqref{33.6}, \eqref{33.7} and  \eqref{33.8}, we deduce
\begin{equation}
\label{33.9}
\int_{B_r\setminus B_{4R}}|(-\D)^\frac{s}{2} u|^2dx\leq Cr^{n-2s}.
\end{equation}
The proof is completed following \eqref{33.5}, \eqref{33.9} and $n>2s.$
\end{proof}

We now apply the estimates in Lemma \ref{le33.2} to prove $L^2$ estimate for $\Delta_bu_e^\l$.
\begin{lemma}\label{le33.3}
We have
\begin{equation*}
\int_{B_r^{n+1}\cap \r}y^{3-2s}|\Delta_bu_e^\l(X)|^2dxdy\leq C(r),\quad \forall r>0,\,\l\geq1.
\end{equation*}
\end{lemma}
\begin{proof}
We write
$$u^\l=u_1^\l+u_2^\l,$$ where
\begin{align*}
u^\l_1(x)&=c(n,s)\int_{\B}\left(\frac{1}{|x-z|^{n-2s}}
-\frac{1}{(1+|z|)^{n-2s}}\right)\varphi(z)e^{u^\l(z)}dz+c_\l,\\
u^\l_2(x)&=c(n,s)\int_{\B}\left(\frac{1}{|x-z|^{n-2s}}
-\frac{1}{(1+|z|)^{n-2s}}\right)(1-\varphi(z))e^{u^\l(z)}dz,
\end{align*}
for $\varphi\in C_c^\infty(B_{4r})$  is such that $\varphi=1$ in $B_{2r}$. As in the proof of Lemma \ref{le33.2} one can show that    $$\int_{\r}y^{3-2s}|\Delta_b\ou^\l_1(X)|^2dxdy=C_{n,s}\int_{\B}\left|\ss u^\l_1(x)\right|^2dx\leq C(r).$$
Here and in the following  $\ou_i^\l$ denotes the $s$-harmonic extension of $u_i^\l,~i=1,2$ through \eqref{maine} respectively. It remains to prove that
\begin{equation}
\label{33.11}
\int_{B_r^{n+1}\cap\r}y^{3-2s}|\Delta_b\ou^\l_2(x)|^2dxdy\leq C(r)\quad \mbox{for every}~r\geq 1,~\lambda\geq1.
\end{equation}
Following the  arguments of Lemma \ref{le31.2} and Lemma \ref{le31.3}, one could verify that
\begin{equation}
\label{33.12}
\|\nabla^2u_2^\l\|_{L^\infty(B_{3r/2})}\leq C(r),\quad
\|\nabla u_2^\l\|_{L^\infty(B_{3r/2})}\leq C(r),\quad
\int_{\B}\frac{|u_{2}^\l(x)|}{1+|x|^{n+2s}}dx\leq C(r),
\end{equation}
and consequently,
\begin{equation}
\label{33.13}
\|u_2^\l\|_{L^\infty(B_{3r/2})}\leq C(r).
\end{equation}
To prove \eqref{33.11}, we notice that
\begin{equation*}
y^{3-2s}|\Delta_b\ou^\l_2(X)|^2\leq C\left(y^{3-2s}|\Delta\ou^\l_2(X)|^2
+y^{1-2s}(\partial_y\ou^\l_{2})^2\right).
\end{equation*}
We first consider the term $y^{1-2s}(\partial_y\ou^\l_{2})^2$, it is known that
\begin{align*}
\partial_y\ou^\l_2(X)=~&\partial_y(\ou^\l_2(x,y)-u^\l_2(x))=\k\partial_y\int_{\B}\frac{y^{2s}}{|(x-z,y)|^{n+2s}}(u_2^\l(z)-u_2^\l(x)-\nabla u_2^\l(x)\cdot(x-z))dz\\
=~&\k\int_{\B}\partial_y\left(\frac{y^{2s}}{|(x-z,y)|^{n+2s}}\right)(u_2^\l(z)-u_2^\l(x)-\nabla u_2^\l(x)\cdot(x-z))dz,
\end{align*}
where we used
$$\k\int_{\B}\frac{y^{2s}}{|(x-z,y)|^{n+2s}}dz=1~\mathrm{and}~
\int_{\B}\frac{y^{2s}(x-z)}{|(x-z,y)|^{n+2s}}dz=0.$$
By \eqref{33.12},  for $|x|\leq r$ it holds that
\begin{equation}
\label{33.14}
\begin{aligned}
&\left|\int_{\B\setminus B_ {3r/2}}\partial_y\left(\frac{y^{2s}}{|(x-z,y)|^{n+2s}}\right) (u_2^\l(z)-u_2^\l(x)-\nabla u_2^\l(x)\cdot(x-z)) dz\right|\\
&\leq C(r)y^{2s-1}\int_{\B\setminus B_ {3r/2}}\frac{(|u_2^\l(z)|+|z|+1)}{1+|z|^{n+2s}}dz\leq C(r)y^{2s-1}.
\end{aligned}
\end{equation}
Using \eqref{33.12}-\eqref{33.13}, we see that
\begin{equation}
\label{33.15}
\begin{aligned}
&\int_{B_r^{n+1}\cap\r}y^{1-2s}\left(\int_{B_{3r/2}}\partial_y\left(\frac{y^{2s}}{|(x-z,y)|^{n+2s}}\right)(u_2^\l(z)-u_2^\l(x)-\nabla u_2^\l(x)\cdot(x-z))dz\right)^2dxdy\\
&\leq C(r)\|\nabla^2 u^\l_2\|_{L^\infty(B_{3r/2})}^2\int_{B_r^{n+1}\cap\r}y^{1-2s}
\left(\int_{B_{3r/2}}\partial_y\left(\frac{y^{2s}}{|(x-z,y)|^{n+2s}}\right)|x-z|^2dz\right)^2dxdy\\
&\leq C(r)\|\nabla^2 u^\l_2\|_{L^\infty(B_{3r/2})}^2\int_{B_r^{n+1}\cap\r}y^{3-2s}dxdy\leq C(r).
\end{aligned}
\end{equation}
By \eqref{33.14}-\eqref{33.15} we get
\begin{equation}
\label{33.16}
\begin{aligned}
\int_{B_r^{n+1}\cap\r}y^{1-2s}|\partial_y\ou^\l_2|^2dxdy
\leq C(r)\int_{B_r^{n+1}\cap\r}\left(y^{3-2s}+y^{2s-1}\right)dxdy\leq C(r).
\end{aligned}
\end{equation}
For the term  $y^{3-2s}\partial_y^2\ou_2^\l$, we have
\begin{equation*}
\partial_y^2\ou_2^\l=\k\int_{\B}\partial_y^2\left(\frac{y^{2s}}{|(x-z,y)|^{n+2s}}\right)(u_2^\l(z)-u_2^\l(x)-\nabla u_2^\l(x)\cdot(x-z))dz.
\end{equation*}
By \eqref{33.12}, for $|x|\leq r$ it holds that
\begin{equation}
\label{33.17}
\begin{aligned}
&\left|\int_{\B\setminus B_ {3r/2}}\partial_y^2\left(\frac{y^{2s}}{|(x-z,y)|^{n+2s}}\right) (u_2^\l(z)-u_2^\l(x)-\nabla u_2^\l(x)\cdot(x-z)) dz\right|\\
&\leq C(r)y^{2s-2}\int_{\B\setminus B_ {3r/2}}\frac{(|u_2^\l(z)|+|z|+1)}{1+|z|^{n+2s}}dz\leq C(r)y^{2s-2}.
\end{aligned}
\end{equation}
Based on \eqref{33.12} and \eqref{33.13}, we get
\begin{equation}
\label{33.18}
\begin{aligned}
&\int_{B_r^{n+1}\cap\r}y^{3-2s}\left(\int_{B_{3r/2}}\partial_y^2\left(\frac{y^{2s}}{|(x-z,y)|^{n+2s}}\right)(u_2^\l(z)-u_2^\l(x)-\nabla u_2^\l(x)\cdot(x-z))dz\right)^2dxdy\\
&\leq C(r)\|\nabla^2 u^\l_2\|_{L^\infty(B_{3r/2})}^2\int_{B_r^{n+1}\cap\r}y^{3-2s}
\left(\int_{B_{3r/2}}\partial_y^2\left(\frac{y^{2s}}{|(x-z,y)|^{n+2s}}\right)|x-z|^2dz\right)^2dxdy\\
&\leq C(r)\|\nabla^2 u^\l_2\|_{L^\infty(B_{3r/2})}^2\int_{B_r^{n+1}\cap\r}y^{3-2s}dxdy\leq C(r).
\end{aligned}
\end{equation}
From \eqref{33.17} and \eqref{33.18} we gain
\begin{equation}
\label{33.19}
\begin{aligned}
\int_{B_r^{n+1}\cap\r}y^{3-2s}|\partial_y^2\ou^\l_2|^2dxdy
\leq C(r)\int_{B_r^{n+1}\cap\r}\left(y^{3-2s}+y^{2s-1}\right)dxdy\leq C(r).
\end{aligned}
\end{equation}
For the terms $y^{3-2s}\partial^2_{x_i}\ou_2^\l,~i=1,\cdots,n$,  in a similar way of deriving \eqref{33.19} we get
\begin{equation}
\label{33.20}
\int_{B_r^{n+1}\cap\r}y^{3-2s}|\Delta_x\ou^\l_2|^2dxdy\leq C(r).
\end{equation}
Then \eqref{33.11} follows from \eqref{33.16}, \eqref{33.19} and \eqref{33.20}. This finishes the proof.
\end{proof}

We also need the following interpolation-type inequality in the study of monotonicity formula. 
\begin{lemma}
\label{le3.interpolation}
Let $w_e^\l$ be the extension of $w^\l$, then we have
\begin{equation*}
\begin{aligned}
\int_{\r\cap B_1^{n+1}}y^{3-2s}|\nabla w_e^\l|^2dxdy
\leq C\left(\int_{\r\cap B_2^{n+1}}y^{3-2s}|\Delta_b w_e^\l|^2dxdy
+\int_{\r\cap B_2^{n+1}}y^{3-2s}|w_e^\l|^2dxdy\right).
\end{aligned}
\end{equation*}
\end{lemma}

\begin{proof}
This is a consequence of \cite[Proposition 2.4]{ff} and we omit the proof.
\end{proof}

Given the priori estimates, we now study asymptotic behavior of the energy functional in the monotonicity formula.

\begin{proposition}
\label{pr33.1}
We have $c_\lambda=O(1)$ for $\lambda\in [1,\infty)$. Moreover,
\begin{equation*}
\lim_{\l\to+\infty}E(\l,0,u_{e})=\lim_{\l\to\infty}E(1,0,u_e^\l)<+\infty.
\end{equation*}
\end{proposition}
\begin{proof}
Applying Theorem \ref{mono}, we know that $E(\l r,0,u_e)=E(r,0,u_e^\l) $ is nondecreasing in $r$. So,
\begin{equation}
\label{th-eq1}
E(1,0,u_e)\leq E(\l,0,u_e)=E(1,0,u_e^\l)
\leq \int_1^2E(t,0,u_e^\l)dt\leq\int_1^2\int_t^{t+1}E(\gamma,0,u_e^\l)d\gamma dt.
\end{equation}
From Proposition \ref{pr32.1} and Lemma \ref{le33.3},  we conclude that
\begin{equation}
\label{th-eq2}
\int_1^2\int_t^{t+1}\gamma^{2s-n}\left(\int_{\r\cap B_\gamma^{n+1}}\frac12y^{3-2s}|\Delta_bu_e^\l|dxdy-C_{n,s}\int_{\br\cap B_\gamma^{n+1}}e^{u_e^\l}dx\right)d\gamma dt\leq C,
\end{equation}
where $C>0$ is independent of $\lambda.$ In addition,
\begin{equation}
\label{th-eq3}
\begin{aligned}
&\int_1^2\int_t^{t+1}\frac{\gamma^3}{2}\frac{d}{d\gamma}\left(\gamma^{2s-3-n}\int_{\r\cap\partial B_\gamma^{n+1}}y^{3-2s}\left(\frac{\partial u_e^\l}{r}+\frac{2s}{\gamma}\right)^2d\sigma\right)d\gamma dt\\
&\leq\frac12\int_1^2(t+1)^{2s-n}\int_{\r\cap\partial B_{t+1}^{n+1}}y^{3-2s}\left(\frac{\partial u_e^\l}{r}+\frac{2s}{t+1}\right)^2d\sigma dt\\
&\quad -\frac12\int_1^2t^{2s-n}\int_{\r\cap\partial B_{t}^{n+1}}y^{3-2s}\left(\frac{\partial u_e^\l}{r}+\frac{2s}{t}\right)^2d\sigma dt\\
&\quad -\frac{3}{2}\int_1^2\int_{t}^{t+1}\left(\gamma^{2s-1-n}\int_{\r\cap\partial B_\gamma^{n+1}}y^{3-2s}\left(\frac{\partial u_e^\l}{r}+\frac{2s}{\gamma}\right)^2d\sigma\right)d\gamma dt\\
&\leq C\int_{\r\cap(B_3^{n+1}\setminus B_1^{n+1})}y^{3-2s}\left(\frac{\partial u_e^\l}{r}+\frac{2s}{r}\right)^2dxdy.
\end{aligned}
\end{equation}
Using Lemma \ref{le3.interpolation} (see \cite{ad} also for the classical case), we get
\begin{equation*}
\begin{aligned}
&\int_{\r\cap B_1^{n+1}}y^{3-2s}|\nabla u_e^\l|^2dxdy=\int_{\r\cap B_1^{n+1}}y^{3-2s}|\nabla w_e^\l|^2dxdy\\
&\leq C\left(\int_{\r\cap B_2^{n+1}}y^{3-2s}|\Delta_b w_e^\l|^2dxdy
+\int_{\r\cap B_2^{n+1}}y^{3-2s}|w_e^\l|^2dxdy\right)\\
&\leq C\left(\int_{\r\cap B_2^{n+1}}y^{3-2s}|\Delta_b u_e^\l|^2dxdy
+\int_{\r\cap B_2^{n+1}}y^{3-2s}|w_e^\l|^2dxdy\right)\leq C.
\end{aligned}
\end{equation*}
Here, Lemma \ref{le33.3} and Lemma \ref{le33.12} are used. Therefore, the right-hand side of \eqref{th-eq3} is bounded. Similarly, the other terms in the monotonicity formula \eqref{energy} can be estimated. Therefore,  \eqref{th-eq1} yields
\begin{equation*}
E(1,0,u_e)\leq E(1,0,u_e^\l)\leq 4s(n+2-2s)c_sc_\l+O(1),
\end{equation*}
which implies $c_\l$ is bounded from below for $\l\geq1.$ Using \eqref{31.7},  we have
\begin{equation*}
u^\l=w^\l+c_\l.
\end{equation*}
It is not difficult to see that $w^\l$ is locally bounded from below by the proof of Lemma \ref{le31.5}, then we conclude from Corollary \ref{cr31.1} that $c_\l\leq C$ for $\l\geq 1$. Thus,   $c_\l$ is bounded. Lemma \ref{le33.1} implies that
\begin{equation*}
\int_{\partial B^{n+1}_1\cap\r}y^{3-2s}u_e^\l(X)d\sigma=O(1).
\end{equation*}
As a result,  from \eqref{th-eq1} we conclude that $E(1,0,u_e^\l)$ is bounded uniformly in $\l$. This completes the proof.
\end{proof}

We end this part with an estimate that is used in the  proof of Theorem \ref{main-th}.
\begin{lemma}
\label{le33.6}
For every $r>1$ and $\l\geq1$,  we have
\begin{equation*}
\int_{B_r^{n+1}\cap\r}y^{3-2s}(|u^\l_e|^2+|\Delta_b u^\l_e|^2)dxdy\leq C(r).
\end{equation*}		
\end{lemma}

\begin{proof}
This is a direct consequence of Lemma \ref{le33.12}, Lemma \ref{le33.3} and $c_\l$ is bounded uniformly for $\l\geq1$.
\end{proof}

\subsection{Proof of the main  result}

\begin{proof}[Proof of Theorem \ref{main-th}.]
Let  $u$  be  finite Morse index solution to \eqref{main} for some  $n>2s$ satisfying  \eqref{stability}. Let $R>1$ be such  that $u$   is stable outside the ball  $B_R$. From Lemma \ref{le33.6} we obtain that there exists a sequence $\lambda_i\to+\infty$ such that $u_e^{\lambda_i}$ converges weakly in $\dot{H}^2_{\mathrm{loc}}(\overline\r,y^{3-2s}dxdy)$ to a function $u_e^\infty$. In addition, we have $u_e^{\l_i}\to u_e^\infty$ almost everywhere. To show that $u^\infty$, the restriction of $u_e^\infty$ in $\rn$, satisfies \eqref{main} in the weak sense, we need to verify the following. First, we need to show that for any $\varepsilon>0$ there exists $r\gg1$ such that
\begin{equation}
\label{4.1}
\int_{\B\setminus B_r}\frac{|u^\l|}{1+|z|^{n+2s}}dz<\varepsilon,\quad\forall\lambda\geq 1.
\end{equation}
Indeed,  as in the proof of Lemma \ref{le31.3}, we get
\begin{equation*}
\begin{aligned}
\int_{\B\setminus B_r}\dfrac{|u^\l(x)|}{1+|x|^{n+2s}}dx\leq~&
C\int_{\B\setminus B_r}\dfrac{c_\lambda}{1+|x|^{n+2s}}dx
+C\int_{\mathbb{R}^n\setminus B_{r_0}}e^{u^\l(z)}|E(z)|dz\\
&+C\int_{\B\setminus B_r}\int_{B_{r_0}}
\frac{e^{u^\l(z)}}{|x-z|^{n-2s}(1+|x|^{n+2s})}dzdx\\
&+C\int_{\B\setminus B_r}\int_{B_{r_0}}
\frac{e^{u^\l(z)}}{((1+|z|)^{n-2s})(1+|x|^{n+2s})}dzdx\\
\leq ~& Cr^{-2s}+Cr_0^{-1/2}+Cr^{-2s}r_0^{n-2s}.
\end{aligned}
\end{equation*}
We could first choose $r_0$ large enough such that $Cr_0^{-1/2}\leq\varepsilon/2$ and then choose $r$ such that $Cr^{-2s}r_0^{n-2s}+Cr^{-2s}\leq\varepsilon/2$. Then \eqref{4.1} is proved. As a consequence, we could show that $u^\infty\in L_s(\B)$, and for any $\varphi\in C_c^\infty(\B)$
\begin{equation*}
\lim\limits_{i\to\infty}\int_{\B}u^{\lambda_i}\s\varphi dx=\int_{\B}u^{\infty}\s\varphi dx.
\end{equation*}
The second point we need to prove is that $e^{u^{\lambda_i}}$ converges to $e^{u^{\infty}}$ in $L_{\mathrm{loc}}^1(\B).$ By \eqref{32.est-1} we can easily see that $e^{u^{\lambda_i}}$ is uniformly integrable in $L_{\mathrm{loc}}^1(\B\setminus\{0\})$. Using  \eqref{32.est-2}, around the origin we get
\begin{equation*}
\int_{B_{\varepsilon}}e^{u^{\lambda_i}}dx=\lambda_i^{2s-n}\int_{B_{\l_i\varepsilon}}e^{u}dx\leq C\varepsilon^{n-2s}.
\end{equation*}
Therefore, we have $e^{u^{\lambda_i}}$ is uniformly integrable in $L_{\mathrm{loc}}^1(\B)$,  and together with $u^{\lambda_i} \to u^\infty$ a.e., we get for any $\varphi\in C_c^\infty(\B)$
\begin{equation*}
\lim\limits_{i\to\infty}\int_{\B}e^{u^{\lambda_i}}\varphi dx=\int_{\B}e^{u^{\infty}}\varphi dx.
\end{equation*}
Then $u^\infty$ is a weak solution to the equation of \eqref{main}. Now  we show that the limit function $u_e^\infty$ is homogenous, and is of the form $-2s\log r+\tau(\theta)$.  Based on the above convergence, we get for any $r>0$,
\begin{equation}
\label{4.r}
{\lim_{i\to\infty}E(\l_ir,0,u_e)~\ \mbox{is independent of}~\ r}.
\end{equation}
Indeed, for any two positive numbers $r_1 <r_2$ we have
\begin{equation*}
\lim_{i\to\infty}E(\l_ir_1,0,u_e)\leq\lim_{i\to\infty}E(\l_ir_2,0,u_e).
\end{equation*}
On the other hand, for any $\lambda_i$, we can choose $\lambda_{m_i}$ such that $\{\lambda_{m_i}\}\subset\{\lambda_i\}$ and $\lambda_i r_2\leq \lambda_{m_i} r_1$. As a consequence, we have
\begin{equation*}
\lim_{i\to\infty}E(\l_ir_2,0,u_e)\leq\lim_{i\to\infty}E(\l_{m_i}r_1,0,u_e)=\lim_{i\to\infty}E(\l_ir_1,0,u_e).
\end{equation*}
This finishes the proof of \eqref{4.r}. Using \eqref{4.r} we see that for $R_2>R_1>0,$
\begin{equation*}
\begin{aligned}
0=~&\lim_{i\to\infty}E(\l_iR_2,0,u_e)-\lim_{i\to+\infty}E(\l_iR_1,0,u_e)\\
=~&\lim_{i\to\infty}E(R_2,0,u_e^{\l_i})-\lim_{i\to\infty}E(R_1,0,u_e^{\l_i})\\
\geq~&C(n,s)\lim_{i\to\infty}\inf\int_{\left(B_{R_2}^{n+1}\setminus B_{R_1}^{n+1}\right)\cap\r}y^{3-2s}r^{2s-2-n}\left(\frac{\partial u_e^{\l_i}}{\partial r}+\frac{2s}{r}\right)^2dxdy\\
\geq~&C(n,s)\int_{\left(B_{R_2}^{n+1}\setminus B_{R_1}^{n+1}\right)\cap\r}y^{3-2s}r^{2s-2-n}\left(\frac{\partial u_e^\infty}{\partial r}+\frac{2s}{r}\right)^2dxdy.
\end{aligned}
\end{equation*}
Notice that in the last inequality we only used the weak convergence of $u_e^{\l_i}$
to $u_e^\infty$ in $H^2_{\mathrm{loc}}(\overline\r,y^{3-2s}dxdy)$. So,
$$\frac{\partial u_e^\infty}{\partial r}+\frac{2s}{r}=0\quad \mbox{a.e. in}\quad \R.$$
Thus we proved the claim. In addition, $u^\infty$ is also stable because the stability condition for $u^{\l_i}$ passes to the limit. Then,  from Theorem \ref{th3.1} we get a contradiction  to \eqref{1.condition}. This completes the proof.
\end{proof}

\begin{remark} It is expected that when (\ref{1.condition}) does not hold that is when
$$\frac{ \Gamma^2(\frac{n+2s}{4})  }{ \Gamma^2(\frac{n-2s}{4}) } \geq \frac{\Gamma(\frac{n}{2}) \Gamma(1+s)}{ \Gamma(\frac{n-2s}{2})}, $$
there exist radial entire stable solutions for \eqref{main}.   The method of construction of such solutions is the one that is applied in \cite{ddw,dggw} and references therein.
\end{remark}

\begin{remark} Consider the case of lower dimensions $n=2s$ that is connected with the conformally invariant equations. When $s=1$ that gives $n=2$, Farina in \cite{fe} constructed radially symmetric solutions which are stable outside a compact set.  For the case of $s=2$ that implies $n=4$, we refer interested readers to  \cite{lin} and references therein where the existence of radial solutions are discussed.
For the case of $s=\frac{3}{2}$ that gives $n=3$, we refer to  \cite{jmmx}.  The arguments in order to show that radially symmetric solutions are stable outside a compact set rely on the Hardy-type inequalities.

\end{remark}

\section{The  fourth order Gelfand-Liouville equation}\label{secfo}
In this section, similar to the case of $1<s<2$, we develop a monotonicity formula for the fourth order (local) Gelfand-Liouville equation
\begin{equation}\label{6.main2}
\Delta^{2} u= e^u \ \ \text{in} \ \ \mathbb{R}^n,
\end{equation}
via technical rescaling arguments. As a  consequence of this formula, we classify finite Morse index solutions applying blow-down analysis arguments and classifications of homogenous solutions. Our proof is structured in variational methods and it is different from the one developed by Dupaigne et al. in \cite{dggw}, see also \cite{ceg} and references therein. In the latter a Kato's inequality and Moser iteration arguments are applied.

\begin{thm}\label{6.mainth}
If $5 \le n\le 12$, then there is no solution $u\in C^{4}(\mathbb R^n)$ of \eqref{6.main2} which is stable outside a compact set and $\Delta u$ vanishes at infinity.
\end{thm}

Note that in lower dimension $n=4$ and higher dimensions $n\ge 13$ it is known that there are solutions of \eqref{6.main2} which are stable outside a compact set, see \cite{bffg}. We refer to \cite{dggw,bffg,ceg,gg} and references therein for more information about the fourth order problem (\ref{6.main2}) and to \cite{ff1} for the case of polyharmonic equation.  The notion of stability for the local problem (\ref{6.main2}) is as what follows.

 \begin{defn}
We say that a solution $u$ of (\ref{6.main2}) is stable outside a compact set  if there exists $R>0$ such that
\begin{equation}\label{6.stability}
 \int_{\mathbb R^n} e^u \phi^2 dx\le \int_{\mathbb R^n}  |\Delta \phi|^2dx,
\end{equation}
 for any $\phi\in C_c^\infty(\mathbb R^n\setminus \overline {B_{R}})$.
 \end{defn}

Following the same proof of Lemma \ref{le32.1}, we could derive that for $t\in[0,1],$
\begin{equation*}
\int_{\mathbb{R}^n}e^{tu}\phi^2dx\leq\int_{\rn}| (-\Delta)^{t} \phi|^2dx,\quad \forall \phi\in C_c^\infty(\rn\setminus \overline {B_{R}}).
\end{equation*}
In particular, for $t=\frac12,$
\begin{equation*}
\label{6.half-stability}
\int_{\mathbb{R}^n}e^{\frac12u}\phi^2dx\leq\int_{\rn}|\nabla\phi|^2dx,\quad \forall \phi\in C_c^\infty(\rn\setminus \overline {B_{R}}).
\end{equation*}
\medskip

Consider homogeneous solutions of the form $u(r,\theta)= -4 \log (r) + \omega(\theta),$
when $\omega:\mathbb S^{n-1}\to\mathbb R$.  Here, we follow the arguments in \cite{fw,hy,ddww} that is inspired by the tangent cone analysis of Fleming \cite{fle}. By the Hardy-Relich inequality the following classification is sharp. We omit the proof.  

\begin{thm}\label{6.thmhom}
Let $n>4$ and $u\in W^{2,2}_{\mathrm{loc}}(\mathbb R^n\setminus \{0\})$ be a stable homogenous solution of \eqref{6.main2} in $\mathbb R^n\setminus \{0\}$. Assume that $e^u\in L^1_{\mathrm{loc}}(\mathbb R^n\setminus \{0\})$. Then,
\begin{equation}\label{6.nroot}
\frac{n^2(n-4)^2}{16} \ge 8 (n-2)(n-4).
\end{equation}
\end{thm}
The inequality (\ref{6.nroot}) implies that $n\ge 13$ that is desired.  We now develop a monotonicity formula for solutions of (\ref{6.main2}). Define
\begin{eqnarray*}\label{6.energy}
E(r,x_0,u)& :=& r^{4-n} \left[   \int_{   B(x_0,r)} \frac{1}{2}  |\Delta u |^2dx-   \int_{   B(x_0,r)} e^{u} dx  \right]
-2r^{ 3-n } \int_{   \partial B(x_0,r)}
 \left(  \frac{\partial u}{\partial r} +\frac{4}{r}\right)^2d\sigma
\\&&
+\frac{1}{2}  r^3      \frac{d}{dr} \left[  r^{ 1-n }    \int_{   \partial B(x_0,r)}  \left(  \frac{\partial u}{\partial r} +\frac{4}{r}    \right )^2  d\sigma   \right]
-8\left(2-n \right)  r^{1-n} \int_{   \partial B(x_0,r)}  \left( u + 4  \log r \right)d\sigma
\\&&
-4\left(2-n \right) r^{  2 -n }  \int_{  \partial B(x_0,r)}     \left(  \frac{\partial u}{\partial r} +\frac{4}{r}    \right)d\sigma + \frac{1}{2}    \frac{d}{dr} \left[ r^{4-n}  \int_{   \partial B(x_0,r)} \left(  | \nabla u|^2 - \left|\frac{\partial u}{\partial r}\right|^2 \right)d\sigma \right]
\\&&+ \frac{1}{2}    r^{3-n}  \int_{   \partial B(x_0,r)}  \left(  | \nabla u|^2 - \left |  \frac{\partial u}{\partial r} \right|^2 \right) d\sigma.
\end{eqnarray*}
For the above energy functional, we have the following result. The proof is similar to the one given for Theorem \ref{mono} and we omit the details.
\begin{thm}\label{6.mono}
Assume that $n>3$. Then, $E(\lambda,x_0,u)$ is a nondecreasing function of $\lambda>0$. Furthermore,
\begin{equation*}
\frac{dE(\lambda,x_0,u)}{d\lambda} \ge 2(n-3)\  \lambda^{2-n}    \int_{\partial B(x_0,\lambda)}  \left( \frac{\partial u}{\partial r} +  \frac{4}{ \lambda}  \right)^2d\sigma.
\end{equation*}
\end{thm}

\subsection{Blow-Down Analysis and Proof of Theorem \ref{6.mainth} }\label{sec5.1}

As discussed in section 2, the rescaled function
$$u^\lambda(x)=u(\l x)+4\log\l$$
is a family of solutions to \eqref{6.main2} and is stable outside a compact set if $u$ is stable outside a compact set. Here is a direct consequence of the stability inequality, the proof follows from the same argument of Lemma \ref{le31.1}
\begin{lemma}
\label{6.cr31.1}	
Suppose $n>4$ and $u$ is a solution of \eqref{6.main2} which is stable outside a compact set. Then there exists $C$ such that
\begin{equation*}
\int_{B_r}e^{u^\l}dx\leq Cr^{n-4},\quad \forall \l\geq1,~r\geq1.
\end{equation*}
\end{lemma}

We now consider the fourth order equation as a system
\begin{equation*}
\begin{cases}
-\Delta u = v &\mathrm{in} \ \ \mathbb R^n,
\\
-\Delta v = e^u &\mathrm{in} \ \ \mathbb R^n,
\end{cases}
\end{equation*}
and apply the Moser iteration type arguments initiated by Crandall and Rabinowitz in \cite{CR}. The following estimates are given in \cite{dggw}.

\begin{lemma} \label{6.le32.4}
Let $u$ be the solution of \eqref{6.main2} with \eqref{6.stability} holds. If there exists a constant $C$ depending only on $\alpha$ such that for $r\geq1$ large enough, we have
\begin{equation}
\label{6.32.condition}
\int_{B_{2r}\setminus B_r}(e^{\alpha u}+v^{2\alpha})dx\leq Cr^{n-4\alpha },\quad  \forall r\geq2R.
\end{equation}
Then
\begin{equation}
\label{6.32.conclusion}
\int_{B_{2r}\setminus B_r}(e^{\frac{n}{n-2}\alpha u}+v^{\frac{2n}{n-2}\alpha})dx\leq Cr^{n-\frac{4n}{n-2}\alpha },\quad \forall r\geq 3R,
\end{equation}
where $\alpha\in(\alpha^\sharp,\min\{\frac{n}{4},\alpha^*\}).$ Here, $\alpha^\sharp,\alpha^*$ denote the largest two roots of the polynomial $X^3-8X+4$.

\end{lemma}

Define the following parameter that depends on the dimension $n$,
\begin{equation}
\bar\alpha=\max\left\{\frac{n}{n-2}\min\left\{\frac{n}{4},\alpha^*\right\},~\min\left\{\frac{n}{4},\alpha^*\right\}+\frac12\right\}.
\end{equation}
Summarizing the previous estimates, we reach the following conclusion for $L^p$ estimates of $e^u$ for certain $p$.

\begin{lemma}
\label{6.pr32.1}
Let $u$ be a solution of \eqref{6.main2}. Assume that $u$ is stable on $\rn\setminus B_R$. Then for every $p\in[1,\bar \alpha)$ there exists $C=C(p)>0$ such that for $r$ large
\begin{equation}
\label{6.32.est-0}
\int_{B_{2r}\setminus B_r}e^{pu}dx\leq Cr^{n-4p}.
\end{equation}
In particular,
\begin{itemize}
\item[(i)] for $|x|$ large,
\begin{align}
\label{6.32.est-1}
\int_{B_{|x|/2}(x)}e^{pu (z)}dz\leq C(p)|x|^{n-4p}, \quad\forall p\in[1,\bar \alpha),
\end{align}
\item[(ii)] for $r$ large
\begin{align}
\label{6.32.est-2}
\int_{B_r\setminus  B_{2R}}e^{pu}dx\leq C(p)r^{n-4p}, \quad\forall p\in[1,\min\{\frac{n}{4},\alpha^*\}).
\end{align}
\end{itemize}
\end{lemma}

Before we give the estimation for each term appeared in the monotonicity formula, we make the following preparation.

\begin{lemma}
\label{6.le-s1}
For $\delta>0$ there exists $C=C(\delta)>0$ such that
\begin{equation*}
\int_{\B}\dfrac{e^{u^\l}}{1+|x|^{n-4+\delta}}dx\leq C,
\quad \forall\lambda\geq1.
\end{equation*}
\end{lemma}

\begin{proof}
By Lemma \ref{6.cr31.1}, we have	
\begin{equation*}
\int_{\rn}\frac{e^{u^\l}}{1+|x|^{n-4+\delta}}dx
\leq C\int_{B_1}e^{u^\l}(x)dx+
\sum_{i=0}^\infty\int_{2^i\leq|x|<2^{i+1}}\frac{e^{u^\l}}{1+|x|^{n-4+\delta}}dx\leq C+C\sum_{i=0}^\infty\frac{1}{2^{i\delta}}<+\infty.
\end{equation*}	
Hence we finish the proof.		
\end{proof}

Similar to the fractional case,  setting
$$w^\l:=c(n)\int_{\rn}\left(\frac{1}{|x-z|^{n-4}}-\frac{1}{(1+|z|)^{n-4}}\right)e^{u^\l(z)}dz,$$
where $c(n)$ is chosen such that
$$c(n)\Delta^2\left(\frac{1}{|x-z|^{n-4}}\right)=\delta(x-z).$$
One can check that $w^\l\in L_{\mathrm{loc}}^1(\rn)$. In addition, following the ideas appeared in the proof of Lemma \ref{le31.3}, Lemma \ref{le31.5} and Lemma \ref{le32.5}, we have the following conclusion.

\begin{lemma}
\label{6.le-s2}	
For any solution of \eqref{6.main2} which is stable outside a compact set and $\Delta u$ vanishes at infinity, we have $w^\l\in L_{2}(\rn)$ and
\begin{equation*}
u^\l(x)=c(n)\int_{\rn}\left(\frac{1}{|x-z|^{n-4}}-\frac{1}{1+|z|^{n-4}}\right)e^{u^\l(z)}
dy+b_\l,
\end{equation*}
for some constant $b_\l.$ In addition, if $n<4\alpha^*$, $u$ has the following asymptotic behavior for $|x|$ large
\begin{equation*}
u(x)\leq-4\log|x|+C.
\end{equation*}
\end{lemma}

Concerning the monotonicity formula, we have
\begin{equation*}
E(\lambda,0,u)=E(1,0,u^\lambda).
\end{equation*}
We now provide an estimation for the boundary integration of $u^\l$.

\begin{lemma}
\label{6.le33.1}
\begin{equation*}
\int_{\partial B_1} u^\l d\sigma=\omega_nb_\l+O(1),
\end{equation*}
where $b_\l$ is introduced in Lemma \ref{6.le-s2} and $\omega_n$ is the surface area of $\mathbb{S}^{n-1}$.
\end{lemma}
\begin{proof}
Using Lemma \ref{6.le-s2} we have
\begin{equation*}
\int_{\partial B_1}u^\l d\sigma=\int_{\partial B_1}w^\l d\sigma+\omega_nb_\l.
\end{equation*}
Using Lemma \ref{6.le-s1} again, we have
\begin{equation*}
\begin{aligned}
\left|\int_{\partial B_1}w^\l d\sigma\right|
&\leq C\int_{\partial B_1}\int_{\rn\setminus B_4}\frac{1+|x|}{|z|^{n-3}}e^{u^\l(z)}dz+
\int_{\partial B_1}\int_{B_4}\left(\frac{1}{(1+|z|)^{n-4}}
+\frac{1}{|x-z|^{n-4}}\right)e^{u^\l(z)}dzd\sigma\\
&\leq C+\int_{\partial B_1}\int_{B_4}\frac{1}{|x-z|^{n-4}}e^{u^\l(z)}dzd\sigma.
\end{aligned}
\end{equation*}
To estimate the second term on the right we notice that
$$\int_{\partial B_1}\frac{1}{|x-z|^{n-4}}d\sigma\leq C,$$
combined with Lemma \ref{6.cr31.1} we have
\begin{equation*}
\int_{\partial B_1}\int_{B_4}\frac{1}{|x-z|^{n-4}}e^{u^\l(z)}dzd\sigma\leq C.
\end{equation*}
Hence we finish the proof.
\end{proof}

For the first quadratic term in $E(r,0,u)$, we have the following estimation

\begin{lemma}
\label{6.le33.3}
We have
\begin{equation*}
\int_{B_r}|v^\l |^2dx =\int_{B_r}|\Delta u^\l |^2dx\leq C r^{n-4},\quad \forall r>0,\,\l\geq1.
\end{equation*}
\end{lemma}

\begin{proof}
The proof is as same as Lemma \ref{le33.2}, we omit the details.
\end{proof}

\begin{proposition}
\label{6.pr33.1}
We have $b_\lambda=O(1)$ for $\lambda\in [1,\infty)$. Moreover,
\begin{equation*}
\lim_{\l\to+\infty}E(\l,0,u)=\lim_{\l\to\infty}E(1,0,u^\l)<\infty.
\end{equation*}
\end{proposition}
\begin{proof}
From the above discussion we can apply a similar arguments in Proposition \ref{pr33.1} to show that
\begin{equation*}
b_\l ~\ \mbox{is bounded from below for}~\ \l\geq1.
\end{equation*}
By Lemma \ref{6.le-s2} we have
\begin{equation*}
u^\l=w^\l+b_\l,
\end{equation*}
with $w^\l$ being a $L^1$ function and bounded from below locally. Using Lemma \ref{6.cr31.1} we derive that $b_\l\leq C$ for $\l\geq 1.$ Thus $b_\l$ is bounded. From Lemma \ref{6.le33.1}, we conclude that
\begin{equation*}
\int_{\partial B_1} u^\l(x)d\sigma=O(1).
\end{equation*}
Then $E(\l,0,u)$ is bounded uniformly in $\l$ and it completes the proof.
\end{proof}

Similar arguments as in Lemma \ref{le33.6} yield the following estimate.

\begin{lemma}
\label{6.le33.6}
For every $r>1$ and $\l\geq1$,  we have
\begin{equation*}
\int_{B_r}(|u^\l|^2+|\Delta u^\l|^2)dx\leq C(r).
\end{equation*}		
\end{lemma}

We  are now ready to give the proof of Theorem \ref{6.mainth}.

\begin{proof}[Proof of Theorem \ref{6.mainth}.]
Let  $u$  be  finite Morse index solution to \eqref{6.main2} for some  $n>4$ satisfying  \eqref{6.stability}. Let $R>1$ be such  that $u$   is stable outside the ball  $B_R$. From the above Lemma \ref{6.le33.6} we obtain that there exists a sequence $\lambda_i\to+\infty$ such that $u^{\lambda_i}$ converges weakly in ${H}^2_{\mathrm{loc}}(\rn)$ to a function $u^\infty$. In addition, we have $u^{\l_i}\to u^\infty$ almost everywhere. As a consequence, for any $\varphi\in C_c^\infty(\B)$
\begin{equation*}
\lim\limits_{i\to\infty}\int_{\B}u^{\lambda_i}\Delta^2\varphi dx=\int_{\B}u^{\infty}\Delta^2\varphi dx.
\end{equation*}
Next, we need to prove that $e^{u^{\lambda_i}}$ converge to $e^{u^{\infty}}$ in $L_{\mathrm{loc}}^1(\B).$ By \eqref{6.32.est-1} we can easily see that $e^{u^{\lambda_i}}$ is uniformly integrable in $L_{\mathrm{loc}}^1(\B\setminus\{0\})$. Using  \eqref{6.32.est-2}, around the origin point we get
\begin{equation*}
\int_{B_{\varepsilon}}e^{u^{\lambda_i}}dx=\lambda_i^{4-n}\int_{B_{\l_i\varepsilon}}e^{u}dx\leq C\varepsilon^{n-4}.
\end{equation*}
Therefore, we have $e^{u^{\lambda_i}}$ is uniformly integrable in $L_{\mathrm{loc}}^1(\B)$,  and together with $u^{\lambda_i} \to u^\infty$ a.e., we get for any $\varphi\in C_c^\infty(\B)$
\begin{equation*}
\lim\limits_{i\to\infty}\int_{\B}e^{u^{\lambda_i}}\varphi dx=\int_{\B}e^{u^{\infty}}\varphi dx.
\end{equation*}
Then $u^\infty$ satisfies \eqref{6.main2} in the weak sense. Now, following the arguments in the proof of Theorem \ref{main-th} one can  show that $u^\infty$ is homogeneous. In addition, $u^\infty$ is also stable in $\rn\setminus\{0\}$ because of the stability condition for $u^{\l_i}$ passing to the limit. Then from Theorem \ref{6.thmhom} we get $n\geq13$, which is a contradiction. This completes the proof.
\end{proof}

\begin{center}
{\bf Acknowledgement}
\end{center}

The research of the second author is partially supported by NSERC of
Canada. The research of the third author is partially supported by NSFC No.11801550 and NSFC No.11871470. The third author is thankful to Ali Hyder for the fruitful discussion.

\end{document}